\documentclass[11pt,a4paper]{article}

\usepackage[T1]{fontenc}
\usepackage{amsmath,amssymb,amsthm,mathtools}
\usepackage[margin=1in]{geometry}
\usepackage{enumitem}
\usepackage{booktabs,tabularx,array}
\usepackage[hidelinks]{hyperref}
\hypersetup{
  pdftitle={Optimising Selberg's method for critical zeros},
  pdfauthor={Andrew Pearce-Crump}
}
\usepackage{graphicx}
\usepackage{tikz}
\usetikzlibrary{arrows.meta}
\definecolor{cSel}{RGB}{33,102,172}
\definecolor{cLev}{RGB}{35,139,69}
\definecolor{cAtk}{RGB}{204,111,0}
\definecolor{cPair}{RGB}{136,65,157}
\usepackage{microtype}

\allowdisplaybreaks
\setlist{nosep}

\theoremstyle{plain}
\newtheorem{theorem}{Theorem}[section]
\newtheorem{proposition}[theorem]{Proposition}
\newtheorem{lemma}[theorem]{Lemma}
\newtheorem{corollary}[theorem]{Corollary}
\theoremstyle{remark}
\newtheorem{remark}[theorem]{Remark}

\newcommand{\Real}{\operatorname{Re}}
\newcommand{\Imag}{\operatorname{Im}}
\newcommand{\half}{\tfrac12}
\newcommand{\Cres}{\mathcal C_{\mathrm{res}}}
\newcommand{\Sdiag}{S_{\!\lambda}(U)}
\newcommand{\dd}{\,d}
\newcommand{\R}{\mathbb R}
\newcommand{\one}{\mathbf 1}
\newcommand{\kappazero}{\kappa_0}

\makeatletter
\renewcommand{\author}[2][]{\gdef\@author{#2}}
\newcommand{\address}[1]{\gdef\@address{#1}}
\newcommand{\email}[1]{\gdef\@email{#1}}
\let\@address\@empty
\let\@email\@empty
\newcommand{\printauthoraddress}{%
  \par\bigskip\noindent\footnotesize
  \ifx\@address\@empty\else\textsc{\@address}\par\fi
  \ifx\@email\@empty\else\smallskip\noindent\textit{Email address:}\ \texttt{\@email}\par\fi}
\makeatother

\title{Optimising Selberg's method for critical zeros}
\author[A. Pearce-Crump]{Andrew Pearce-Crump}
\address{School of Mathematics, University of Bristol, Bristol, BS8 1UG, United Kingdom}
\email{andrew.pearce-crump@bristol.ac.uk}
\date{}

\begin{document}
\maketitle

\begin{abstract}
We revisit Zhuravlev's 1974 quantitative form of Selberg's sign-change method for the zeros of the Riemann zeta function on the critical line.  Zhuravlev's work, originally in Russian and little known in the West, established an explicit positive proportion of such zeros.  Using modern techniques we optimise the arithmetic mean value at the heart of the method and obtain the strongest result yet known from Selberg's method, proving that more than $6.78\%$ of the non-trivial zeros of the Riemann zeta function lie on the critical line.
\end{abstract}

\section{Introduction}
\label{sec:intro}

Let \(N(T)\) denote the number of non-trivial zeros
\(\rho=\beta+i\gamma\) of the Riemann zeta function \(\zeta(s)\) with
\(0<\gamma\le T\), counted with multiplicity, and let \(N_0(T)\) count those
with \(\beta=\tfrac12\).  We write
\[
 \kappa_0=\liminf_{T\to\infty}\frac{N_0(T)}{N(T)}.
\]
Hardy proved that infinitely many zeros of the zeta function lie on the
critical line, and Hardy--Littlewood obtained \(N_0(T)\gg T\)
\cite{Hardy,HardyLittlewood}.  Selberg's theorem \(N_0(T)\gg T\log T\) was the
first positive-proportion result \cite{Selberg,Titchmarsh}: his method detects
sign changes of Hardy's real function after multiplication by a non-negative
mollifier.  Levinson's argument-principle method and Atkinson's two-point sign
test are analytically different; Section~\ref{sec:history} compares the three
approaches, as well as the recent pair correlation method, and places
Zhuravlev's work within the Selberg tradition.

This paper has two aims.  First, we give a corrected and self-contained account
of Zhuravlev's 1974 paper \cite{Zhuravlev}, which is written in Russian and
little known in the West.  Second, we optimise the arithmetic at the heart of
his and Selberg's method.  Our main result is the following.

\begin{theorem}\label{thm:main}
More than \(6.78\%\) of the non-trivial zeros of \(\zeta(s)\) lie on the critical line; more precisely, \(\kappa_0>0.0678190\).
\end{theorem}

\begin{remark}
This proportion does not approach the strongest bounds available for
$\zeta(s)$.  The value of the present paper lies elsewhere: it gives a method
for producing a positive proportion of zeros on the critical line for
$L$-functions where the other methods may not be available, and it optimises
the arithmetic on which that method rests.
\end{remark}

\subsection{Selberg's method}
\label{sec:intro-method}

Selberg's method rests on an elementary observation.  On the critical line there
is a real-valued function \(Z(t)\), Hardy's function, whose real zeros are
exactly the ordinates of the zeros of \(\zeta(s)\) on the line.  A sign change of
\(Z(t)\) between two ordinates forces a critical zero of odd order between them, so
the number of sign changes on \([T,2T]\) is a lower bound for the number of
critical zeros there.  Counted directly this bound is weak.  Hardy and
Hardy--Littlewood obtained in this way only \(N_0(T)\gg T\)
\cite{Hardy,HardyLittlewood}, a density-zero set.

Selberg's device is to multiply \(Z(t)\) by \(|M(\tfrac12+it)|^2\), where \(M(s)\)
is a Dirichlet polynomial called a mollifier.  Since \(|M(s)|^2\ge0\), the product
keeps the sign changes of \(Z(t)\) away from the zeros of \(M(s)\), which are of even
order and harmless, so it detects the same critical zeros.  The gain is
analytic, since the mollifier can be chosen to flatten the large fluctuations of
\(Z(t)\), sharpening the mean-value inequalities that turn average size into a
count of sign changes.  A second moment of the mollified detector, passed
through Jensen's inequality, then yields a \emph{positive} proportion of sign
changes, so that \(N_0(T)\gg T\log T\).  How large a proportion is governed by
the mean value of the mollified detector, hence by the coefficients of \(M(s)\);
choosing them well is the arithmetic problem at the heart of the method.

Zhuravlev made this quantitative for an arbitrary Dirichlet polynomial.  We
write $U$ for the mollifier length, where Zhuravlev writes $Z$, so that $Z(t)$
is free for Hardy's function throughout.  For
real coefficients \(b_n\), put
\[
 M_U(s)=\sum_{n\le U}b_n n^{-s},
 \qquad
 Y_U(s)=\zeta(s)M_U(s)M_U(1-s).
\]
Here $\aleph(s)=\chi(s)^{-1/2}$ is the unimodular factor defined
in~\eqref{eq:aleph} below, where $\chi(s)$ is the factor in the functional
equation $\zeta(s)=\chi(s)\zeta(1-s)$; it satisfies $|\aleph(\half+it)|=1$ and
is what makes $\aleph(s)Y_U(s)$ real on the critical line.  On the critical
line, \(\aleph(s)Y_U(s)\) is Hardy's function multiplied by
\(|M_U(s)|^2\), and so has the same sign changes as Hardy's function away from
its zeros.  Lavrik's approximate functional equation supplies an analytic half
\(\beta_U\) with
\begin{equation}\label{eq:intro-c2}
 \aleph(s)Y_U(s)=2\operatorname{Re}\beta_U(s)
 \qquad (\operatorname{Re}s=\tfrac12),
\end{equation}
whose mean square is governed by
\begin{equation}\label{eq:intro-Sb}
 S_b(1,U)=\sum_{d\le U^2}\varphi(d)
 \left(\sum_{\substack{\ell,m\le U\\d\mid\ell m}}
 \frac{b_\ell b_m}{\ell m}\right)^2.
\end{equation}
If \(U=T^\theta\), \(0<\theta<\tfrac14\), and
\(S_b(1,U)\le(\gamma+o(1))/\log U\), then under the normalisation
\(\aleph(s)Y_U(s)=c\operatorname{Re}\beta_U(s)\),
\begin{equation}\label{eq:intro-master}
 \kappa_0\ge \frac{2\theta}{e c^2\gamma}.
\end{equation}
Proposition~\ref{prop:distilled} proves this, and Section~\ref{sec:Zsource}
verifies its hypotheses uniformly in the coefficients used below.

Zhuravlev's coefficient choice gives \(\gamma_{\mathrm{Zh}}=3\pi^2/16\).  The
source states \eqref{eq:intro-c2} but its later logarithmic inequality is
written as if the factor \(2\) were absent, and its polar term is printed with
denominator \(1-s\) in place of \(s-1\).  Correcting both points and carrying
the missing term \(-T\log2\) through the optimisation gives the benchmark
\begin{equation}\label{eq:intro-Zh-corrected}
 \kappa_0\ge \frac{1}{8e\gamma_{\mathrm{Zh}}}
 =\frac{2}{3\pi^2e}
 =0.0248493202\ldots,
\end{equation}
rather than the printed proportion \(2/21\).

\subsection{Summary of the paper}
\label{sec:intro-summary}

We optimise the arithmetic form \eqref{eq:intro-Sb}.  Writing
\(\zeta(s)^{-1/2}=\sum_{n\ge1}\lambda(n)n^{-s}\), the sharp coefficients
\(b_n=\lambda(n)\) give
\begin{equation}\label{eq:intro-sharp}
 S_\lambda(U)=\frac{\mathcal C_{\mathrm{res}}}{\log U}
 +o\!\left(\frac1{\log U}\right),
 \qquad
 \mathcal C_{\mathrm{res}}
 =\frac{\Gamma(\tfrac14)^4}{2\pi^4}
 =\frac{2}{\Gamma(\tfrac34)^4},
\end{equation}
and hence
\begin{equation}\label{eq:intro-sharp-kappa}
 \kappa_0\ge\frac{\Gamma(\tfrac34)^4}{16e}=0.0518466520\ldots.
\end{equation}
For the logarithmic power profiles \(b_{n,a}(U)=\lambda(n)(\log(U/n)/\log U)^a\) with \(a>0\), an explicit non-negative polyhedral constant \(\mathcal C_a\) arises; at \(a=\tfrac14\), rigorous interval computation gives \(0.6779054<\mathcal C_{1/4}<0.6786576\), and hence \(\kappa_0>0.0677586\).  Finally, combining the powers \(v^{1/4}\) and \(v\) gives the profile of Section~\ref{sec:psd} and Theorem~\ref{thm:main}, while a positive-semidefinite sum of squared mollifiers gives nothing further.  Thus the sharp, quarter-power and optimised detectors improve the corrected Zhuravlev benchmark by factors of about \(2.09\), \(2.727\) and \(2.729\).

These factors are measured against Zhuravlev's benchmark, which rests on
Wirsing's mean-value theorem; his coefficients, evaluated exactly rather than
bounded, already give $0.8052\ldots$ (Section~\ref{sec:Zh-exact}).  Most of the
improvement therefore lies in evaluating the diagonal rather than in the choice
of coefficients, and what this paper adds to the arithmetic is the profile: the
multiplicative structure of $\lambda$ is already optimal, by the local
inequality at the end of Section~\ref{sec:psd}.

The proof separates into analytic and arithmetic parts.  The analytic reconstruction uses the sign-preserving Selberg detector, Lavrik's approximate functional equation, a rectangle inequality, and a coefficient-uniform rational-frequency off-diagonal estimate.  The arithmetic analysis begins with the convolution identity \(\lambda*\lambda=\mu\), removes the same-side poles from a four-variable Euler product, and evaluates the resulting bipartite quarter-order residue exactly; the profile optimisation is then a finite computation certified in Appendix~\ref{app:psd}.

Section~\ref{sec:history} gives the historical comparison.  Sections~\ref{sec:sourceguide}--\ref{sec:Zsource} reconstruct and correct Zhuravlev's analytic argument.  Sections~\ref{sec:sharp}--\ref{sec:transfer} evaluate the sharp diagonal, Section~\ref{sec:weights} treats the power profiles, and Section~\ref{sec:psd} treats sums of squared mollifiers and carries out the profile optimisation, with the resulting critical-line proportions at its end.  Appendix~\ref{app:Lavrik-horizontal} supplies the analytic large-disk bounds and Appendix~\ref{app:psd} the finite-dimensional certificates.

\section{Historical context: Hardy, Selberg, Levinson, Atkinson and pair \mbox{correlation}}
\label{sec:history}

Before giving the full details of the argument, we explore the various
methods that have been used to produce values of $\kappa_0$, and where the
present method sits among them.

\subsection{Hardy's function and sign changes}

Write the functional equation of $\zeta(s)$ as
\[
 \zeta(s)=\chi(s)\zeta(1-s),
 \qquad
 \chi(s)=\pi^{s-1/2}
 \frac{\Gamma((1-s)/2)}{\Gamma(s/2)},
\]
and choose a continuous square root on the critical line.  Hardy's function
\[
 Z(t)=\chi(\tfrac12+it)^{-1/2}\zeta(\tfrac12+it)
\]
is real for real \(t\), and its real zeros are precisely the ordinates of
zeros of \(\zeta(s)\) on \(\operatorname{Re}s=\tfrac12\).  Every sign change
therefore detects a zero of odd multiplicity.  Hardy's theorem and the
Hardy--Littlewood lower bound are based on this principle, supplemented by
mean-value estimates for \(Z(t)\) \cite{Hardy,HardyLittlewood,Titchmarsh}.
We recall how the principle is applied, since the same shape
of argument recurs throughout.  Hardy's original argument compares two
averages of $Z(t)$ over $[T,2T]$.  The mean $\int_T^{2T}Z(t)\dd t$ is small,
because $Z(t)$ oscillates, while the mean of $|Z(t)|$, bounded below through
$\int_T^{2T}Z(t)^2\dd t\asymp T\log T$ and Cauchy--Schwarz, is not.  If
$Z(t)$ had constant sign on a long subinterval the two averages would agree
there in size, so $Z(t)$ must change sign, and a positive number of times.
Hardy--Littlewood turned this into the quantitative count $N_0(T)\gg T$ by
running the comparison on each of $\asymp T$ short intervals.  The whole
difficulty is that the second moment, which supplies the lower bound, is
dominated by the rare large values of $Z(t)$ rather than by its typical size,
so the comparison is lossy; Selberg's mollifier is precisely a device for
flattening those large values before the comparison is made.
A sign-change count is at most \(N_0(T)\), but it has the additional feature
that it counts only zeros of odd multiplicity.

\subsection{Selberg's sign-preserving mollifier}

Selberg multiplied the real detector by a non-negative mollifier.  If
\(M(s)\) is a Dirichlet polynomial, then on the critical line
\[
 \chi(s)^{-1/2}\zeta(s)M(s)M(1-s)
 =Z(t)|M(\tfrac12+it)|^2.
\]
Zeros of \(M(s)\) occur with even order in this real function and do not create
false sign changes.  A mollifier may therefore reduce the mean size of the
detector without changing its sign.  Selberg used the coefficients of
\(\zeta(s)^{-1/2}\), with a linear logarithmic weight, and combined the
mollified mean square with a Littlewood-type rectangle argument to prove a
positive proportion of critical-line zeros \cite{Selberg,Titchmarsh}.

The first explicit numerical value of Selberg's implicit constant was obtained by S.-H. Min in his 1947 Oxford D.Phil. dissertation \cite{MinThesis}.  Working from Titchmarsh's simplification of Selberg's argument, Min proved that
\[
 \kappa_0>\frac{1}{14{,}074{,}731};
\]
see also \cite[\S10.9, pp.~268--269]{Titchmarsh}.

The quantitative problem splits into three parts: a sign-change inequality,
a mean-value theorem for an analytic half of the mollified functional
equation, and an arithmetic optimisation of the mollifier coefficients.
Zhuravlev's paper is a particularly explicit implementation of this scheme.
It allows general initial coefficients, reduces the mean square to the form
\(S_b(1,U)\), and then chooses one multiplicative sequence to obtain a
numerical constant.

\subsection{Levinson's argument-principle method}

Levinson introduced a different zero detector.  One forms a differential
combination such as
\[
 G(s)=\zeta(s)+\frac{\zeta'(s)}{\log T},
\]
or, more generally,
\(V(s)=Q(-\log(T)^{-1}d/ds)\zeta(s)\), and applies the argument principle on
a line
\(\sigma_0=\tfrac12-R/\log T\).  A mollifier approximating \(1/\zeta(s)\)
is then used to estimate a twisted second moment of \(V(s)\).  This counts zeros
through complex-analytic variation on a displaced line, rather than through
sign changes of a real function.
In more detail, the point of the differential combination is that
$V(s)$ inherits the zeros of $\zeta(s)$ that lie on the critical line but has
its own zeros pushed off it in a controlled way.  Counting the zeros of
$V(s)$ in the rectangle $\sigma_0<\Re s<\tfrac12$, $T<\Im s<2T$ by the
argument principle, and converting that count by Littlewood's lemma, replaces
it by $\frac1{2\pi}\int_T^{2T}\log|V(\sigma_0+it)|\dd t$ up to a term the
functional equation supplies exactly.  Every zero of $\zeta(s)$ off the
critical line contributes to the rectangle count, so an upper bound for the
integral is a lower bound for the proportion on the line.  The integral is
estimated by Jensen's inequality from the mollified second moment
$\int_T^{2T}|V(\sigma_0+it)\psi(\sigma_0+it)|^2\dd t$, where $\psi(s)$ is a
Dirichlet polynomial approximating $\zeta(s)^{-1}$; the mollifier is what
keeps that mean square bounded, and its length is what limits the resulting
proportion.

Levinson proved that more than one third of the zeros lie on the critical
line \cite{Levinson}.  Conrey obtained more than two fifths
\cite{Conrey}, and later multi-piece and higher-order mollifiers pushed the
proportion beyond \(41\%\) \cite{BuiConreyYoung}.  The strongest bound from the
mollifier methods is the theorem of Pratt, Robles, Zaharescu and Zeindler that
more than five twelfths of the zeros lie on the line \cite{PRZZ}.

\subsection{Atkinson's two-point detector}

Atkinson's 1948 argument is a third method \cite{Atkinson}.  It uses Hardy's
real function at two nearby points.  The displacement is chosen through the
Hardy phase, and the inequality
\[
 Z(t)Z(t')<0
\]
forces a zero between \(t\) and \(t'\).  The measure of the negative set is
controlled by a shifted mixed mean of \(Z(t)Z(t')\) and by Ingham's fourth
moment \cite{Atkinson,Ingham}.  Baluyot later combined this detector with a
mollifier and a twisted fourth-moment formula to give another proof of a
positive proportion \cite{Baluyot}.  His Theorem~1.1 is stated qualitatively,
but the closing section of that paper optimises the argument explicitly and
gives \(\kappa_0>0.0001049\), rising to \(\kappa_0>0.0086729\) if his
fourth-moment estimates are granted for mollifiers of any length.

Selberg and Atkinson are both sign-change methods, but their statistics are
different: Selberg uses a one-point mollified second moment, whereas Atkinson
uses a two-point correlation and a fourth-moment majorant.  Levinson's method
is different from both, since it applies the argument principle to a
mollified differential combination off the critical line.

\subsection{The pair-correlation method}
\label{sec:paircorr}

The three methods above are local.  Each forms a mollified detector, a
Dirichlet polynomial times \(\zeta(s)\) or a differential combination of
\(\zeta(s)\), and reads critical zeros off a first, second or fourth moment of
that detector.  A fourth route is global, and works through the pair
correlation of the zeros.  Montgomery's 1973 study of the pair correlation, via
Weil's explicit formula, showed under the Riemann hypothesis that at least two
thirds of the zeros are simple.  The explicit formula supplies a Hermitian form
\[
 W(f,g)=\sum_\rho m_\rho\,\widehat f(\gamma_\rho)\widehat g(\gamma_\rho)
\]
over the zeros, whose positivity on every test function is equivalent to the
Riemann hypothesis.  Montgomery's evaluation of the associated prime-side second
moment is a mean value of a Dirichlet polynomial of length \(T\) and is
unconditional; the hypothesis entered only to read the zero side termwise as a
positive sum over real ordinates.

That last step was made unconditional in 2026, raising the record for the
proportion of zeros on the critical line to two thirds.  The proof was
discovered autonomously by Claude, an AI model developed by Anthropic, and was
verified and communicated by Alp\"oge and Furman \cite{AlpogeFurman}.
The idea is to replace the single Hermitian form $W$ by a finite matrix.  One
takes a family of about $N(T,2T)$ modulated copies of one fixed window,
equispaced through $[T,2T]$, and assembles the values of the explicit formula
on them into a finite real symmetric matrix.  On the zero side the matrix is a
sum of two pieces: each zero actually on the critical line contributes a
positive rank-one piece, while each pair of zeros off the line contributes a
piece with one positive and one negative direction.  Hence a matrix with few
positive directions can accommodate only few on-line zeros, and the classical
inertia and rank--trace inequalities of linear algebra make this quantitative.
On the prime side the same matrix has a size that Montgomery's mean value
pins down unconditionally.  Comparing the two sides bounds the number of
on-line zeros from below, and the strength of the bound depends only on the
choice of window: the indicator window gives two thirds, and the
Montgomery--Taylor window gives $0.6725$.  What matters for us is that the
positivity is used globally, over the whole family of zeros at once, rather
than locally through a detector.

Shortly after, Lamzouri \cite{Lamzouri} obtained the same bound by a shorter
argument.  Replacing the finite-dimensional matrix framework above by a single
inequality in the Hilbert space of the explicit formula, he applies Montgomery's
pair-correlation theorem directly.

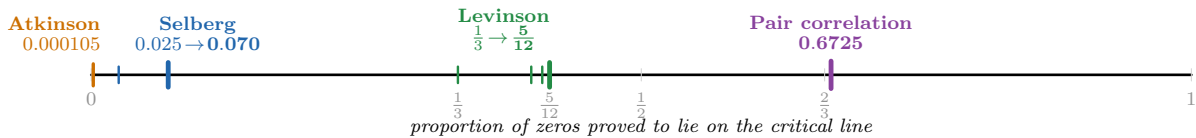
\begin{figure}[t]
\centering
\resizebox{\textwidth}{!}{%
\footnotesize
\begin{tikzpicture}[line cap=round]
  \draw[black,line width=1.1pt] (0,0) -- (16,0);
  \foreach \x in {0,5.3333,6.6667,8.0,10.6667,16.0}{\draw[gray!35] (\x,-0.09) -- (\x,0.09);}
  \foreach \x/\lab in {0/{$0$},5.3333/{$\tfrac13$},6.6667/{$\tfrac{5}{12}$},8.0/{$\tfrac12$},10.6667/{$\tfrac23$},16.0/{$1$}}
     {\node[gray!85,anchor=north,font=\scriptsize] at (\x,-0.15) {\lab};}
  \node[anchor=north,font=\scriptsize\itshape] at (8.0,-0.52)
       {proportion of zeros proved to lie on the critical line};
  \draw[cAtk,line width=1.4pt] (0.03,-0.16) -- (0.03,0.16);
  \node[cAtk,anchor=south east,align=right,font=\scriptsize] at (0.25,0.24)
       {\textbf{Atkinson}\\[-1pt] $0.000105$};
  \draw[cSel,line width=1pt] (0.40,-0.12) -- (0.40,0.12);
  \draw[cSel,line width=2pt] (1.085,-0.18) -- (1.085,0.18);
  \node[cSel,anchor=south,align=center,font=\scriptsize] at (1.55,0.24)
       {\textbf{Selberg}\\[-1pt] $0.025\!\to\!\mathbf{0.068}$};
  \draw[cLev,line width=1pt] (5.3333,-0.12) -- (5.3333,0.12);
  \draw[cLev,line width=1pt] (6.40,-0.12) -- (6.40,0.12);
  \draw[cLev,line width=1pt] (6.56,-0.12) -- (6.56,0.12);
  \draw[cLev,line width=2pt] (6.6667,-0.18) -- (6.6667,0.18);
  \node[cLev,anchor=south,align=center,font=\scriptsize] at (6.0,0.24)
       {\textbf{Levinson}\\[-1pt] $\tfrac13\!\to\!\mathbf{\tfrac{5}{12}}$};
  \draw[cPair,line width=2pt] (10.76,-0.20) -- (10.76,0.20);
  \node[cPair,anchor=south,align=center,font=\scriptsize] at (10.76,0.26)
       {\textbf{Pair correlation}\\[-1pt] $\mathbf{0.6725}$};
\end{tikzpicture}%
}
\caption{Lower bounds on the proportion of nontrivial zeros proved to lie on the
critical line, shown on a single axis for comparison.  Each method appears in its
own colour, with its current record in bold and earlier records as thin ticks.
Atkinson's constant is explicit but very small, so its tick sits
at the origin at this scale; Baluyot's optimisation of the constant in
Atkinson's method gives \(0.000105\), rising
to \(0.0087\) if his fourth-moment input is granted for mollifiers of any
length.}
\label{fig:proportions}
\end{figure}

\section{Zhuravlev's argument and the corrections required}
\label{sec:sourceguide}

Unlike Levinson and Atkinson, Zhuravlev did not introduce a new detector.  His
contribution was to carry Selberg's scheme through to an explicit constant for a
general Dirichlet polynomial, by the analytic-half and mean-value steps outlined
in Section~\ref{sec:intro-method}.  This section and the two that follow
reconstruct that implementation, identify the normalisation discrepancy behind
the corrected benchmark, and prove the coefficient-uniform estimates needed for
the corrected theorem.  We follow Zhuravlev's four-stage architecture, but
distinguish quoted inputs from corrected or independently reproved statements.

\subsection{Structure of the argument}

For a Dirichlet polynomial \(M_U(s)\), form
\(
 Y_U(s)=\zeta(s)M_U(s)M_U(1-s).
\)
On the critical line its sign is that of Hardy's function.  Lavrik's
approximate functional equation splits \(Y_U(s)\) into two conjugate halves and
defines an analytic half $\Lambda_U(s)$, a Dirichlet series in the
generalised frequencies of $Y_U(s)$, of which $Y_U(s)$ is twice the real part
after multiplication by the unimodular factor $\aleph(s)$; we write
$\beta_U(s)=\aleph(s)\Lambda_U(s)$ for that rotated half, so that
$\aleph(s)Y_U(s)=2\Real\beta_U(s)$ on the critical line.  A rectangle form of Littlewood's lemma,
followed by Jensen's inequality, bounds sign changes in terms of the second
moment of \(\beta_U(s)\) on
\(\operatorname{Re}s=\tfrac12-x/2\), where $x=1-2\delta$ measures how far
that line has been displaced to the left of the critical line.  The diagonal part of this moment is
then expressed through \(S_b(1,U)\).  The first three stages are analytic;
the last is the arithmetic optimisation addressed in the second half of this
paper.

\subsection{Zhuravlev's coefficient choice}
\label{sec:Zhcoeff}

Zhuravlev introduces
\[
 \beta_d=\frac{f(d)}{N},
 \qquad
 N=\sum_{d\le U}\mu(d)f(d),
\]
where \(f\) is squarefree-supported and multiplicative.  For primes larger
than \(10\),
\(
 f(p)=-1/(2(p-1)),
\)
while
\[
 f(2)=-\frac14,
 \qquad f(3)=-\frac{17}{100},
 \qquad f(5)=-\frac1{10},
 \qquad f(7)=-\frac7{100}.
\]
The coefficients \(b_\ell\) are recovered from
\[
 \frac{b_\ell}{\ell}
 =\sum_{d\le U}\mu\!\left(\frac d\ell\right)\beta_d.
\]
We shall need the following elementary bounds in the analytic argument.
Because \(f\) is squarefree-supported, is negative at every prime, and is
multiplicative,
\[
 \mu(d)f(d)=|f(d)|,\qquad
 N=\sum_{d\le U}|f(d)|,\qquad 1\le N\le U.
\]
If \(\ell\) is not squarefree, then \(b_\ell=0\).  If \(\ell\) is squarefree,
write \(d=\ell q\) in the preceding inversion formula.  Then
\[
 \frac{b_\ell}{\ell}
 =\frac{f(\ell)}{N}
  \sum_{\substack{q\le U/\ell\\(q,\ell)=1}}|f(q)|,
 \qquad
 |b_\ell|\le \ell|f(\ell)|
 =\prod_{p\mid\ell}p|f(p)|\le1.
\]
Here \(p|f(p)|\le0.51\) for \(p=2,3,5,7\), while
\(p|f(p)|=p/(2(p-1))\le11/20\) for \(p\ge11\).  In particular, if
\(U>10\) is prime, only \(d=U\) contributes to the terminal coefficient and
\begin{equation}\label{eq:Zh-terminal}
 \frac{b_U}{U}=\frac{f(U)}{N}
 =-\frac{1}{2(U-1)N},
 \qquad
 |b_U|\ge\frac1{2U},
 \qquad
 \log^+\frac1{|b_U|}\ll\log U.
\end{equation}
A Möbius decomposition gives
\[
 S_b(1,U)\le N^{-4}\sum_{d\le U^2}\varphi(d)F(d)^2,
\]
where \(\varphi(d)F(d)^2\) is multiplicative and, with
\(u_p=|f(p)|\),
\begin{align*}
 \varphi(p)F(p)^2&=(p-1)(2u_p+3u_p^2)^2,\\
 \varphi(p^2)F(p^2)^2&=p(p-1)u_p^4,
\end{align*}
with no higher prime-power terms.

\begin{proposition}\label{prop:Zh-arithmetic}
The preceding coefficients satisfy
\begin{equation}\label{eq:Zh-gamma}
 S_b(1,U)\le
 \left(\frac{3\pi^2}{16}+o(1)\right)\frac1{\log U}.
\end{equation}
\end{proposition}

\begin{proof}
Applying Wirsing's mean-value theorem \cite{Wirsing} with exponents
\(1/2\) and \(1\), respectively, gives
\[
 S_b(1,U)\le \frac{\pi^2}{8\log U}
 \left\{\prod_p R_p(u_p)+o(1)\right\},
\]
where
\[
 R_p(u)=
 \frac{1+(p-1)(2u+3u^2)^2+p(p-1)u^4}
 {(1+u)^4(1-p^{-1})}.
\]
For \(p>10\), direct simplification at
\(u=1/(2(p-1))\) gives
\[
 R_p(u)\le1+\frac{17}{16(p-1)^2}.
\]
A finite rational calculation gives
\[
 \prod_{p\in\{2,3,5,7\}}R_p(u_p)
 \prod_{11\le p\le31}
 \left(1+\frac{17}{16(p-1)^2}\right)<1.446.
\]
For the remaining primes,
\[
 \log\prod_{p>31}
 \left(1+\frac{17}{16(p-1)^2}\right)
 \le \frac{17}{16}\sum_{n\ge32}\frac1{(n-1)^2}
 <\frac{17}{480},
\]
so the tail is less than \(e^{17/480}<1.037\).  Hence
\(\prod_pR_p(u_p)<3/2\), which proves \eqref{eq:Zh-gamma}.
\end{proof}

The shape of these coefficients is not arbitrary, and neither is the loss
they appear to incur.  Section~\ref{sec:Zh-exact} shows that
$pf(p)\to-\tfrac12$ is forced by the power of $\log U$; that the restriction to
squarefree support costs under $5\%$; that Zhuravlev's normalisation $N$
silently supplies a logarithmic profile which is worth a factor $0.81$; and
that his coefficients, evaluated by the machinery of
Sections~\ref{sec:sharp}--\ref{sec:weights} instead of by Wirsing's inequality,
already give $\kappa_0\ge0.057$.  Most of the distance between
\eqref{eq:Zh-gamma} and the results below is therefore in the mean-value
estimate rather than in the coefficients, which are close to the best available
in their class.

\subsection{Corrections to the printed argument}
\label{sec:corrections}

There are three issues relevant to the proof.
\begin{enumerate}[label=\textup{(\roman*)}]
 \item The source states
 \(\aleph(s)Y_U(s)=2\operatorname{Re}\beta_U(s)\), but its later logarithmic
 inequality is written as if the factor \(2\) were absent.  The missing term
 is \(-T\log2\).  At fixed arithmetic input this divides the exact unrounded
 output \(8/(3\pi^2e)\) by four and gives \(2/(3\pi^2e)\).
 Zhuravlev subsequently rounded the uncorrected output down to \(2/21\);
 consequently the corrected exact value is not literally \((2/21)/4\).

 \item The secondary term in the printed truncated-diagonal formula is not
 the coefficient-uniform identity for arbitrary \(b_n\).  Direct expansion
 gives the non-negative quadratic form \(Q_b(2\delta,U)\) in
 Lemma~\ref{lem:truncated-diagonal}.  Its coefficient is
 \(\zeta(2\delta)<0\), which is exactly what the upper bound requires.

 \item Several exponents and an integration endpoint in the printed
 off-diagonal formula are indistinct.  We therefore do not reconstruct a
 theorem by conjecturally correcting that display.  Proposition~\ref{prop:Zhflex}
 uses instead an independent spacing argument for the rational frequencies.
\end{enumerate}
The original copy also resolves a number of local typographical readings,
including the denominator \(\Gamma(s/2)(s-1)\) in Lavrik's polar term and the
prime \(p=U\) in the horizontal Jensen argument.  These local corrections do
not alter the leading constant.

There is one further editorial correction to the electronic source material that we have of the original paper.
In the facsimile on original p.~44, and in Lavrik's defining integral, the
kernel is \(X^{-z}\).  An earlier electronic transcription and translation
omitted the minus sign and displayed \(X^z\).  We use the facsimile reading
\(X^{-z}\) below.  This is a transcription error, not an error in
Zhuravlev's printed paper.

\subsection{Original and modern ingredients}

The zero detector, Littlewood's rectangle argument, Lavrik's approximate
functional equation, Wirsing's theorem, and direct mean-value expansions were
all available in Zhuravlev's period.  Our reconstruction changes neither the
detector nor its conceptual basis.  The modern elements are the
normalisation-invariant formulation \eqref{eq:intro-master}, the
coefficient-uniform rational-frequency proof of the off-diagonal estimate
(Section~\ref{sec:Zsource}), the four-variable Euler-product and
polyhedral analysis (Sections~\ref{sec:euler}--\ref{sec:residue}), the
uniform complex-power Selberg--Delange transfer
(Section~\ref{sec:transfer}), and the rigorous interval certification of
the weighted constant (Section~\ref{sec:weights} and
Appendix~\ref{app:psd}).  These tools make the method easier
to verify and expose the arithmetic optimisation problem more clearly, but
they should not be confused with a new zero-detection principle.

\section{Zhuravlev's sign-change reduction}
\label{sec:Zhreduction}
\label{sec:sign}

We now give the part of Zhuravlev's argument that converts a mollified
mean-square estimate into critical-line zeros.  This section is independent of
our later choice of coefficients.

Let
\[
 M_U(s)=\sum_{n\le U}b_n n^{-s},
 \qquad b_n\in\mathbb R,
 \qquad b_1=1,
 \qquad |b_n|\le1.
\]
Also put
\[
 Y_U(s)=\zeta(s)M_U(s)M_U(1-s).
\]
Since the two mollifier factors are interchanged by $s\mapsto1-s$, $Y_U$ has
the same functional equation as $\zeta$:
\[
 Y_U(s)=\chi(s)Y_U(1-s).
\]
Set
\begin{equation}\label{eq:aleph}
 \aleph(s)=\chi(s)^{-1/2}
 =\left\{\pi^{1/2-s}
 \frac{\Gamma(s/2)}{\Gamma((1-s)/2)}\right\}^{1/2},
\end{equation}
with the branch chosen continuously in the region under consideration.  On
$\operatorname{Re}s=\tfrac12$ one has $|\aleph(s)|=1$ and
\[
 \aleph(s)Y_U(s)
 =\aleph(s)\zeta(s)|M_U(s)|^2\in\mathbb R.
\]
Consequently its sign changes occur only at zeros of $\zeta$ of odd
multiplicity; zeros of $M_U$ enter through $|M_U|^2$ and have even order.

Expand $Y_U$ as a generalised Dirichlet series
\[
 Y_U(s)=\sum_{\xi>0}a_\xi\xi^{-s},
\]
where the frequencies are positive rationals.  Lavrik's symmetric approximate
functional equation, in the form used by Zhuravlev, constructs an analytic
function $\Lambda_U(s)$ representing one half of this functional equation.
We write
\[
 \beta_U(s)=\aleph(s)\Lambda_U(s).
\]
The two properties needed below are
\begin{equation}\label{eq:critical-real-part}
 \aleph(s)Y_U(s)=2\operatorname{Re}\beta_U(s)
 \qquad (\operatorname{Re}s=\tfrac12)
\end{equation}
and, for $s=\sigma+it$ with $t\asymp T$,
\begin{equation}\label{eq:sign-AFE-preview}
 \Lambda_U(s)=
 \sum_{\xi<\sqrt{t/(2\pi)}}a_\xi\xi^{-s}
 +O\!\left(t^{-\sigma/2}U\log^3T\right).
\end{equation}
Section~\ref{sec:Zsource} states and proves the uniform form of
\eqref{eq:sign-AFE-preview}, together with the horizontal argument bounds and
the off-diagonal estimate.  Here we show how these inputs are used.

Let $A(T,2T)$ denote the number of distinct ordinates in $[T,2T]$ at which
$\zeta(\tfrac12+it)$ has a zero of odd multiplicity.  Then
$A(T,2T)\le N_0(2T)-N_0(T)$.

\begin{lemma}
\label{lem:rectangle-sign}
Let $0<\delta<\tfrac12$, put $x=1-2\delta$, and suppose that
\[
 \aleph(s)Y_U(s)=c\operatorname{Re}\beta_U(s)
 \qquad (\operatorname{Re}s=\tfrac12),
 \qquad c>0.
\]
Assume that the net horizontal argument increments of $\beta_U$ in the
rectangle
$\delta\le\operatorname{Re}s\le\tfrac12$,
$T\le\operatorname{Im}s\le2T$, and the corresponding auxiliary boundary
increments, contribute $o(T)$.  Then
\begin{equation}\label{eq:rectangle-sign}
 \frac{\pi x}{2}A(T,2T)
 \ge -\int_T^{2T}\log|\beta_U(\delta+it)|\,dt
      -T\log c-o(T).
\end{equation}
\end{lemma}

\begin{proof}
For a function $h$ analytic in a rectangle
$\sigma_1\le\operatorname{Re}s\le\sigma_2$,
$T\le\operatorname{Im}s\le2T$, Littlewood's rectangle argument gives
\begin{align}\label{eq:Littlewood-inequality}
 \int_T^{2T}&\log|h(\sigma_1+it)|\,dt
 +\int_{\sigma_1}^{\sigma_2}\arg h(\sigma+2iT)\,d\sigma\notag\\
 &\qquad\ge
 \int_T^{2T}\log|h(\sigma_2+it)|\,dt
 +\int_{\sigma_1}^{\sigma_2}\arg h(\sigma+iT)\,d\sigma,
\end{align}
with the usual limiting convention if a zero lies on the boundary.

On the critical line, the hypothesis of the lemma gives
\[
 |\beta_U(\tfrac12+it)|
 \ge |\operatorname{Re}\beta_U(\tfrac12+it)|
 =\frac1c|\aleph(\tfrac12+it)Y_U(\tfrac12+it)|.
\]
Since $|\aleph|=1$ there,
\begin{equation}\label{eq:critical-log-lower}
 \int_T^{2T}\log|\beta_U(\tfrac12+it)|\,dt
 \ge
 \int_T^{2T}\log|\zeta(\tfrac12+it)|\,dt
 +2\int_T^{2T}\log|M_U(\tfrac12+it)|\,dt
 -T\log c.
\end{equation}
Apply \eqref{eq:Littlewood-inequality} separately to $\zeta$ and $M_U$ in
the rectangle $\tfrac12\le\operatorname{Re}s\le3$.  On the right edge,
$|\zeta(3+it)-1|$ and $|M_U(3+it)-1|$ are both at most
$\zeta(3)-1<1$.  Their principal logarithms therefore have absolutely
convergent Dirichlet expansions with zero constant term, so their vertical
integrals are $O(\log T)$ and $O(1)$.  The corresponding horizontal
argument increments are $O(\log T)$ for $\zeta$ and $O(U)$ for $M_U$: the
real and imaginary parts of $M_U(\sigma+it)$ are exponential polynomials in
$\sigma$ with at most $U$ terms.  It follows from
\eqref{eq:critical-log-lower} that
\begin{equation}\label{eq:critical-beta-lower}
 \int_T^{2T}\log|\beta_U(\tfrac12+it)|\,dt
 \ge -T\log c-o(T).
\end{equation}

It remains to control the horizontal boundary terms when
\eqref{eq:Littlewood-inequality} is applied to $h=\beta_U$ between $\delta$
and $\tfrac12$.  Along the critical edge, the real part of $\beta_U$ is the
real sign-detecting function $\aleph(s)Y_U(s)$.  Divide $[T,2T]$ at the sign changes of this real part and choose the
argument continuously from the lower endpoint, with the usual infinitesimal
detours around zeros.  On each intervening interval the curve
$\beta_U(\tfrac12+it)$ remains in one closed half-plane, so its argument can
change by at most $\pi$.  Consequently, for every $u\in[T,2T]$,
\[
 \left|\arg\beta_U(\tfrac12+iu)-\arg\beta_U(\tfrac12+iT)\right|
 \le \pi\{A(T,u)+1\}.
\]
Zeros of $M_U(s)$ do not create additional sign changes, because they occur in
$\aleph(s)Y_U(s)=Z(t)|M_U(s)|^{2}$
with even order.  Transporting this common branch
from the critical edge to the two horizontal edges changes it by only the
assumed $o(T)$ boundary increments.  Hence the difference of the two
horizontal argument integrals in \eqref{eq:Littlewood-inequality} is at most
$\pi(\tfrac12-\delta)A(T,2T)+o(T)$.  Using this and
\eqref{eq:critical-beta-lower} gives
\[
 \pi\left(\frac12-\delta\right)A(T,2T)
 \ge -\int_T^{2T}\log|\beta_U(\delta+it)|\,dt
      -T\log c-o(T),
\]
which is \eqref{eq:rectangle-sign} because
$\tfrac12-\delta=x/2$.
\end{proof}

Zhuravlev's arithmetic form is
\begin{equation}\label{eq:Sb}
 S_b(1,U)=\sum_{d\le U^2}\varphi(d)
 \left(\sum_{\substack{\ell,m\le U\\d\mid\ell m}}
 \frac{b_\ell b_m}{\ell m}\right)^2.
\end{equation}
The following elementary lemma records the passage from dyadic intervals to
the global lower proportion.

\begin{lemma}\label{lem:dyadic-passage}
Suppose that, for some $q\ge0$,
\[
 N_0(2T)-N_0(T)\ge q\{N(2T)-N(T)\}+o(T\log T)
\]
as $T\to\infty$.  Then $\kappa_0\ge q$.
\end{lemma}

\begin{proof}
Set $D(X)=N_0(X)-qN(X)$.  Applying the hypothesis with $T=X/2$ gives
\[
 D(X)\ge D(X/2)-\varepsilon(X)X\log X,
 \qquad \varepsilon(X)\to0.
\]
Iteration down to a fixed height yields
\[
 D(X)\ge -\sum_{j\ge0}
 \varepsilon(X/2^j)\frac{X}{2^j}\log X+O(1).
\]
For any fixed $J$, the first $J$ terms are $o(X\log X)$; the remaining
geometric tail is $O(2^{-J}X\log X)$.  Letting first $X\to\infty$ and then
$J\to\infty$ gives $D(X)=o_-(X\log X)$.  Since
$N(X)\sim X\log X/(2\pi)$, the result follows.
\end{proof}

The next proposition isolates the optimisation step at the heart of the
method.

\begin{proposition}
\label{prop:distilled}
Let $U\asymp T^\theta$, $\log U=\theta\log T+O(1)$, $0<\theta<1/4$, and
$x=1-2\delta\asymp1/\log T$.  Suppose
\[
 \aleph(s)Y_U(s)=c\operatorname{Re}\beta_U(s)
 \qquad (\operatorname{Re}s=\tfrac12),
 \qquad c>0,
\]
the horizontal errors in Lemma~\ref{lem:rectangle-sign} are $o(T)$, and,
uniformly in this moving range,
\begin{equation}\label{eq:mean-general}
 \frac1T\int_T^{2T}|\beta_U(\delta+it)|^2\,dt
 \le \frac{\gamma}{x\log U}(1+o(1)).
\end{equation}
Then
\[
 A(T,2T)\ge\frac{T\log U}{\pi e c^2\gamma}(1+o(1)),
 \qquad
 \kappa_0\ge\frac{2\theta}{e c^2\gamma}.
\]
\end{proposition}

\begin{proof}
By concavity of the logarithm and Cauchy's inequality,
\[
 \int_T^{2T}\log|\beta_U(\delta+it)|\,dt
 \le \frac T2\log\left(
 \frac1T\int_T^{2T}|\beta_U(\delta+it)|^2\,dt\right).
\]
The critical-line identity contributes $-T\log c$ to the rectangle lemma,
so using \eqref{eq:mean-general} gives
\[
 A(T,2T)\ge
 \frac{T}{\pi x}
 \left\{\log\frac{x\log U}{\gamma}-2\log c\right\}
 +o(T/x).
\]
Put $u=x\log U/\gamma$.  The function
$\log(u/c^2)/u$ is maximised at $u=ec^2$.  Taking
\[
 x=\frac{ec^2\gamma}{\log U}
\]
therefore yields
\[
 A(T,2T)\ge\frac{T\log U}{\pi ec^2\gamma}(1+o(1)).
\]
Finally,
\[
 N(2T)-N(T)\sim\frac{T}{2\pi}\log T,
 \qquad
 \log U=\theta\log T,
\]
so the same lower proportion holds on every sufficiently large dyadic
interval.  Lemma~\ref{lem:dyadic-passage}, with
$q=2\theta/(ec^2\gamma)$, now gives
\[
 \liminf_{T\to\infty}\frac{N_0(T)}{N(T)}
 \ge\frac{2\theta}{ec^2\gamma}.
\]
\end{proof}

The original Russian source has $c=2$.  Its equation~(5) omits the resulting
term $-T\log2$ and consequently optimises as though $c=1$.  The remainder of
the paper computes improved arithmetic constants $\gamma$ and verifies the
analytic hypotheses independently of this correction.

\section{Coefficient-uniform analytic estimates}
\label{sec:Zsource}

This section verifies the analytic input required by
Proposition~\ref{prop:distilled}.  We take Lavrik's general theorem on
analytic halves of Dirichlet functional equations as a quoted foundational
input, and derive from it the coefficient-uniform approximate functional
equation and the horizontal Jensen bound used by Zhuravlev.  For the mean
square we give a new proof of the
off-diagonal estimate, based only on the spacing of the rational frequencies.
This avoids any dependence on the typographically ambiguous exponents in
Zhuravlev's printed equation~(8), whose exact reading is not needed for the
reconstructed proof.

Let
\begin{equation}\label{eq:alambda-explicit}
 a_\xi=\sum_{\substack{n\ge1,\ \ell,m\le U\\n\ell/m=\xi}}
 \frac{b_\ell b_m}{m},
 \qquad
 Y_U(s)=\sum_{\xi>0}a_\xi\xi^{-s}\quad(\Real s>1),
\end{equation}
We now make the analytic half, and its dependence on the smoothing parameter,
explicit.  On a dyadic block \(T\le\Imag s\le2T\), fix
\[
 \rho_T^+=\exp\!\left\{i\left(\frac\pi2-\frac1T\right)\right\},
 \qquad
 \rho_T^-=\overline{\rho_T^+}.
\]
These parameters remain fixed while \(s\) moves in the rectangle; this point
is needed for analyticity.  For \(|\arg X^2|<\pi/2\), put
\begin{align}
 J(s,X)&=\frac1{2\pi i}\int_{\Real z=\Delta}
  \Gamma\!\left(\frac{s+z}{2}\right)\frac{X^{-z}}{z}\,dz,
  \qquad \Delta>\max(0,\Real s),\label{eq:Lavrik-J}\\
 F(s,X)&=\frac{J(s,X)}{\Gamma(s/2)}
 =\frac{\Gamma(s/2,X^2)}{\Gamma(s/2)}.\label{eq:Lavrik-F}
\end{align}
The last equality follows first for \(\Real s>0\) by differentiating in
\(X\), and elsewhere by analytic continuation.  Notice the minus sign in
\(X^{-z}\); it is visible in both Lavrik's formula and Zhuravlev's
facsimile.  Define the paired analytic halves
\begin{equation}\label{eq:Lavrik-half-explicit}
 \Lambda_U^\pm(s)
 =\sum_{\xi>0}a_\xi\xi^{-s}
 F\!\left(s,(\pi\rho_T^\pm)^{1/2}\xi\right)
 +\frac{\pi^{1/2}(\pi\rho_T^\pm)^{(s-1)/2}M_U(0)M_U(1)}
 {\Gamma(s/2)(s-1)}.
\end{equation}
We write \(\Lambda_U=\Lambda_U^+\) and
\(\beta_U(s)=\aleph(s)\Lambda_U^+(s)\) on the positive block.  Since the coefficients
are real,
\[
 \Lambda_U^-(\sigma-iu)
 =\overline{\Lambda_U^+(\sigma+iu)}.
\]
Lavrik's fundamental identity
\cite[Lemma~2.1, English p.~133]{Lavrik}, specialised to the functional
equation of \(Y_U\), gives
\[
 \aleph(s)Y_U(s)=2\Real\{\aleph(s)\Lambda_U^+(s)\}
 \qquad(\Real s=\tfrac12,\ T\le\Imag s\le2T).
\]
Lavrik's Theorem~1
\cite[English pp.~130--131]{Lavrik} is applied below with its required
ordinate-dependent phase.  Lemma~\ref{lem:fixed-phase-comparison} compares
that pointwise half with the fixed analytic half above.  Appendix
\ref{app:Lavrik-horizontal} derives the larger-disk bounds required for
Jensen directly from \eqref{eq:Lavrik-F}; no fixed-strip theorem is used
outside its stated range.

Put
\begin{equation}\label{eq:Rb}
 R_b(U)=\left(\sum_{\ell,m\le U}\frac{|b_\ell b_m|}{m}\right)^2.
\end{equation}
If $|b_n|\le1$, then
\begin{equation}\label{eq:Rb-bound}
 R_b(U)\ll U^2\log^2(2U).
\end{equation}

\begin{lemma}\label{lem:transition-kernel}
Uniformly for $A\ge2$ and $h>0$,
\[
 \sum_{n\ge1}
 \frac{\exp\{-(hn)^2/A^2\}}
 {1+|A-\pi(hn)^2/A|}
 \ll (1+h^{-1})\log(2A).
\]
\end{lemma}

\begin{proof}
Put $y_0=A/\sqrt\pi$.  On $hn\le y_0/2$ and $hn\ge2y_0$ the
denominator is $\gg A+(hn)^2/A$, and comparison with an integral gives
$O(1+h^{-1})$.  In the remaining range,
\[
 1+|A-\pi(hn)^2/A|\asymp1+|hn-y_0|.
\]
Divide this range into the sets $j\le |hn-y_0|<j+1$.  Each contains
$O(1+h^{-1})$ integers $n$, and summing $(1+j)^{-1}$ for $j\ll A$ gives
the stated logarithm.
\end{proof}

\begin{lemma}\label{lem:fixed-phase-comparison}
Let \(s=\sigma+it\), where \(T\le t\le2T\) and
\(0\le\sigma\le1/2\), and put
\[
 \rho(t)=\exp\!\left\{i\left(\frac\pi2-\frac1t\right)\right\}.
\]
Let \(\widetilde\Lambda_{U,t}(s)\) be the expression in
\eqref{eq:Lavrik-half-explicit} with \(\rho_T^+\) replaced by \(\rho(t)\).
If \(b_1=1\) and \(|b_n|\le1\) for \(n\le U\), then
\begin{equation}\label{eq:fixed-phase-comparison}
 \Lambda_U(s)-\widetilde\Lambda_{U,t}(s)
 \ll t^{-\sigma/2}U\log(2U),
\end{equation}
uniformly in the displayed range.
\end{lemma}

\begin{proof}
Write \(a=s/2\), \(r=\pi\xi^2\),
\(\phi_T=\pi/2-1/T\), and \(\phi_t=\pi/2-1/t\).  The two incomplete-gamma
endpoints \(r e^{i\phi_T}\) and \(r e^{i\phi_t}\) are joined by the circular
arc \(\mathcal A_r\).  Hence
\[
 F(s,(\pi\rho_T^+)^{1/2}\xi)
 -F(s,(\pi\rho(t))^{1/2}\xi)
 =-\frac1{\Gamma(a)}\int_{\mathcal A_r}e^{-u}u^{a-1}\,du .
\]
On this arc,
\[
 |\phi_t-\phi_T|\le T^{-1},\qquad
 \cos(\arg u)\gg t^{-1},\qquad
 \exp\!\left\{t\left(\frac\pi4-\frac{\arg u}{2}\right)\right\}\ll1.
\]
Uniform Stirling in \(0\le\sigma\le1/2\) therefore gives
\begin{equation}\label{eq:fixed-phase-kernel}
 \left|F(s,(\pi\rho_T^+)^{1/2}\xi)
 -F(s,(\pi\rho(t))^{1/2}\xi)\right|
 \ll t^{-1/2-\sigma/2}\xi^\sigma e^{-c\xi^2/t}.
\end{equation}
Indeed, the arc length is \(r|\phi_t-\phi_T|\), while
\(|u^{a-1}|=r^{\sigma/2-1}e^{-t\arg u/2}\) and
\[
 |\Gamma(a)|^{-1}\ll t^{1/2-\sigma/2}e^{\pi t/4}.
\]

Expanding \(a_\xi\) by \eqref{eq:alambda-explicit} and using
\(\sum_{n\ge1}e^{-A n^2}\ll1+A^{-1/2}\), we obtain
\begin{align*}
 \sum_{\xi>0}|a_\xi|e^{-c\xi^2/t}
 &\le \sum_{\ell,m\le U}\frac1m\sum_{n\ge1}
 e^{-c(n\ell/m)^2/t}\\
 &\ll \sum_{\ell,m\le U}\left(\frac1m+\frac{\sqrt t}{\ell}\right)
 \ll U\sqrt t\log(2U).
\end{align*}
Thus \eqref{eq:fixed-phase-kernel}, after multiplication by
\(\xi^{-\sigma}\), contributes
\(O(t^{-\sigma/2}U\log(2U))\).

It remains to compare the polar terms.  On the interval between
\(\phi_T\) and \(\phi_t\), differentiation of
\(e^{i\phi(s-1)/2}\), followed by the same Stirling estimate and
\(|M_U(0)M_U(1)|\ll U\log(2U)\), gives
\[
 \frac{\sqrt\pi\,|M_U(0)M_U(1)|}{|\Gamma(s/2)(s-1)|}
 \left|e^{i\phi_T(s-1)/2}-e^{i\phi_t(s-1)/2}\right|
 \ll t^{-1/2-\sigma/2}U\log(2U).
\]
This is smaller than the asserted bound and proves the lemma.
\end{proof}

\begin{lemma}
\label{lem:coefAFE}
Let $b_n\in\mathbb R$, $b_1=1$, and $|b_n|\le1$ for $n\le U$.
Let $T\le t\le2T$, $0\le\sigma\le1/2$, and $U<T^{1/4}$.  With
\[
 \tau(t)=\sqrt{t/(2\pi)},
 \qquad
 \Lambda_U^0(s;t)=\sum_{\xi<\tau(t)}a_\xi\xi^{-s},
\]
Lavrik's approximate-functional-equation function satisfies
\begin{equation}\label{eq:coefAFE}
 \Lambda_U(s)=\Lambda_U^0(s;t)
 +O\!\left(t^{-\sigma/2}U\log^3(2T)\right),
\end{equation}
uniformly in the displayed range.  Consequently, if
$\beta_U(s)=\aleph(s)\Lambda_U(s)$ and $\delta=\Real s$, then
\begin{equation}\label{eq:betaAFE}
 \beta_U(\delta+it)=
 \left(\frac{t}{2\pi}\right)^{\delta/2-1/4}
 \omega_\delta(t)\Lambda_U^0(\delta+it;t)
 +O\!\left(t^{-1/4}U\log^3(2T)\right),
\end{equation}
where $\omega_\delta(t)$ is smooth and
\[
 |\omega_\delta(t)|=1+O(T^{-1}),
 \qquad \frac{d}{dt}|\omega_\delta(t)|^2\ll T^{-2}.
\]
\end{lemma}

\begin{proof}
At the single point \(s=\sigma+it\), apply Lavrik's Theorem~1
\cite[English pp.~130--131]{Lavrik} to
\(\widetilde\Lambda_{U,t}(s)\).  Its parameter \(\rho(t)\) has exactly the
phase required in that theorem.  This pointwise comparison is used only for
the approximate functional equation; the fixed analytic function
\(\Lambda_U\) remains the one used in the rectangle and Jensen arguments.
The theorem's hypotheses hold because
\(Y_U(s)=\zeta(s)M_U(s)M_U(1-s)\) has the stated Riemann-type functional
equation, its analytic half has only the inherited pole at \(s=1\), and
\eqref{eq:alambda-explicit} converges absolutely in the initial half-plane.
The theorem gives
\begin{align*}
 \widetilde\Lambda_{U,t}(s)={}&\sum_{\xi<\tau(t)}a_\xi\xi^{-s}
 +O\!\left(|M_U(0)M_U(1)|t^{\sigma/2-1/2}\right)\\
 &+O\!\left(t^{-\sigma/2}
 \sum_{\xi>0}|a_\xi|e^{-\xi^2/t}
 \left(1+\left|\sqrt t-\frac{\pi\xi^2}{\sqrt t}\right|\right)^{-1}
 \right).
\end{align*}
Lemma~\ref{lem:fixed-phase-comparison} changes the left side back to
\(\Lambda_U(s)\) at a cost \(O(t^{-\sigma/2}U\log(2U))\), which is absorbed
by the final error below.
The first error is $O(t^{-\sigma/2}U\log(2U))$, since
$|M_U(0)|\le U$, $|M_U(1)|\ll\log(2U)$, and $\sigma\le1/2$.
For the second, expand $a_\xi$ using \eqref{eq:alambda-explicit} and apply
Lemma~\ref{lem:transition-kernel} with $A=\sqrt t$ and $h=\ell/m$:
\begin{align*}
 &\sum_{\xi>0}|a_\xi|e^{-\xi^2/t}
 \left(1+\left|\sqrt t-\frac{\pi\xi^2}{\sqrt t}\right|\right)^{-1}\\
 &\quad\le
 \sum_{\ell,m\le U}\frac1m
 \sum_{n\ge1}
 \frac{e^{-(n\ell/m)^2/t}}
 {1+|\sqrt t-\pi(n\ell/m)^2/\sqrt t|}\\
 &\quad\ll \log(2T)
 \sum_{\ell,m\le U}\left(\frac1m+\frac1\ell\right)
 \ll U\log(2U)\log(2T).
\end{align*}
This is stronger than the asserted logarithmic bound.  Finally, Stirling's
formula gives
\[
 \aleph(\delta+it)=
 \left(\frac{t}{2\pi}\right)^{\delta/2-1/4}
 \omega_\delta(t),
 \qquad
 |\omega_\delta(t)|=1+O(T^{-1}),\quad
 \frac{d}{dt}|\omega_\delta(t)|^2\ll T^{-2}.
\]
The Stirling remainder multiplying the main polynomial is absorbed by the
same elementary coefficient summation used above.  This proves
\eqref{eq:betaAFE}.
\end{proof}

We use $\Delta_C\arg f$ for the net increment of a continuously chosen
argument along an oriented curve $C$.  It is not the total variation.

\begin{lemma}
\label{lem:coefvariation}
Let \(U\asymp T^\theta\) be prime, where \(0<\theta<1/4\) is fixed, and
suppose
\[
 b_1=1,\qquad |b_n|\le1\ (n\le U),\qquad b_U\ne0.
\]
Put
\[
 L=\log(2T),\qquad \eta_U=\log^+\frac1{|b_U|}.
\]
Assume in addition that \(\eta_U\ll L\), with an implied constant independent
of \(T\).  This condition holds for every coefficient family used below.
For fixed $0\le\sigma_1<\sigma_2\le3$ and $T\le t\le2T$,
\begin{equation}\label{eq:coefvariation}
 \left|\Delta_{[\sigma_1+it,\sigma_2+it]}\arg\Lambda_U\right|
 +\left|\Delta_{[\sigma_1+it,\sigma_2+it]}\arg\beta_U\right|
 \ll ZL\{L^2+\eta_U+1\}.
\end{equation}
\end{lemma}

\begin{proof}
Use the paired analytic halves in \eqref{eq:Lavrik-half-explicit}.  For
\(s=\sigma+it\),
\(\Lambda_U^-(\sigma-it)=\overline{\Lambda_U^+(\sigma+it)}\).  Hence the
zeros of the real and imaginary parts of \(\Lambda_U^+(\sigma+it)\) on the
horizontal segment are respectively the zeros there of
\[
 G_+(s)=\frac{\Lambda_U^+(s)+\Lambda_U^-(s-2it)}2,
 \qquad
 G_-(s)=\frac{\Lambda_U^+(s)-\Lambda_U^-(s-2it)}{2i}.
\]
Choose
\[
 \sigma_0=A U\{L^2+\eta_U+1\}+\sigma_1+\sigma_2,
 \qquad s_0=\sigma_0+it,
\]
where \(A\) is an absolute sufficiently large constant.
Proposition~\ref{prop:Lavrik-Jensen-bounds}, proved from the explicit
incomplete-gamma kernel in Appendix~\ref{app:Lavrik-horizontal}, gives
\[
 \max_{|s-s_0|\le2\sigma_0}\log^+
 \left|(s-1)(s-1-2it)G_\pm(s)\right|
 \ll ZL\{L^2+\eta_U+1\}.
\]
The same proposition shows that at least one choice of sign satisfies
\[
 -\log\left|(s_0-1)(s_0-1-2it)G_\pm(s_0)\right|
 \ll ZL\{L^2+\eta_U+1\}.
\]
The two linear factors remove the pole of \(\Lambda_U^+(s)\) at \(s=1\)
and that of \(\Lambda_U^-(s-2it)\) at \(s=1+2it\).
Jensen's theorem between the concentric radii \(\sigma_0\) and
\(2\sigma_0\) therefore bounds the number of zeros of the chosen \(G_\pm\)
in the inner disk by \(O(ZL\{L^2+\eta_U+1\})\).  The horizontal segment lies
in that disk.  Between two consecutive zeros of the chosen real or imaginary
part, a continuous argument of \(\Lambda_U^+\) changes by less than \(\pi\).
This proves the first term in \eqref{eq:coefvariation}.  Finally,
\(\beta_U=\aleph\Lambda_U^+\), and uniform Stirling gives a horizontal
argument increment \(O(\log T)\) for \(\aleph\); the second term follows.
\end{proof}

For $0<\alpha\le1$, define
\[
 S_b(\alpha,U)=\sum_{d\le U^2}\Phi_\alpha(d)
 \left(\sum_{\substack{\ell,m\le U\\d\mid\ell m}}
 \frac{b_\ell b_m}{\ell m}\right)^2,
 \qquad
 \Phi_\alpha(d)=d^\alpha\prod_{p\mid d}(1-p^{-\alpha}).
\]
Thus $S_b(1,U)$ is \eqref{eq:Sb}.

\begin{lemma}
\label{lem:truncated-diagonal}
Let $1/4\le\delta<1/2$, put $\alpha=2\delta$ and $x=1-\alpha$.
Uniformly for $X\ge2$,
\begin{align}\label{eq:truncated-diagonal}
 \sum_{\xi\le X}\frac{a_\xi^2}{\xi^{\alpha}}
 ={}&\frac{X^x}{x}S_b(1,U)
 +\zeta(\alpha)Q_b(\alpha,U)
 +O\!\left(X^{-\alpha}R_b(U)\right),
\end{align}
where
\begin{equation}\label{eq:Qb}
 Q_b(\alpha,U)=
 \sum_{d\le U^2}\Phi_\alpha(d)
 \left(
  \sum_{\substack{\ell,m\le U\\d\mid\ell m}}
  \frac{b_\ell b_m}{\ell^\alpha m}
 \right)^2\ge0.
\end{equation}
Since $\zeta(\alpha)<0$ for $0<\alpha<1$, it follows that
\begin{equation}\label{eq:truncated-diagonal-upper}
 \sum_{\xi\le X}\frac{a_\xi^2}{\xi^{2\delta}}
 \le\frac{X^x}{x}S_b(1,U)
 +O\!\left(X^{-2\delta}R_b(U)\right).
\end{equation}
\end{lemma}

\begin{proof}
Expand the square using \eqref{eq:alambda-explicit}.  The equality
$n_1\ell_1/m_1=n_2\ell_2/m_2$ is equivalent to
$n_1A=n_2B$, where $A=\ell_1m_2$ and $B=\ell_2m_1$.
Writing $g=(A,B)$, $A=ga$, $B=gb$, one has
$n_1=bk$, $n_2=ak$, and the common frequency is
$kq$, where $q=\ell_1\ell_2/g$.  Euler summation, uniformly for
$1/2\le\alpha<1$ and $Y>0$ (with an empty sum when $Y<1$), gives
\[
 \sum_{k\le Y}k^{-\alpha}
 =\frac{Y^{1-\alpha}}{1-\alpha}+\zeta(\alpha)+O(Y^{-\alpha}).
\]
After multiplication by $q^{-\alpha}$ and summation over
$\ell_i,m_i$, the coefficient of $X^x/x$ is
\[
 \sum_{\ell_1,m_1,\ell_2,m_2\le U}
 \frac{b_{\ell_1}b_{m_1}b_{\ell_2}b_{m_2}}
      {\ell_1m_1\ell_2m_2}
 (\ell_1m_2,\ell_2m_1).
\]
Interchanging $m_1$ and $m_2$ and using
$\sum_{d\mid n}\varphi(d)=n$ identifies this with $S_b(1,U)$.
The coefficient of $\zeta(\alpha)$ is
\[
 \sum_{\ell_1,m_1,\ell_2,m_2\le U}
 \frac{b_{\ell_1}b_{m_1}b_{\ell_2}b_{m_2}
       (\ell_1m_2,\ell_2m_1)^\alpha}
      {\ell_1^\alpha\ell_2^\alpha m_1m_2}.
\]
After the same interchange of $m_1,m_2$ and the identity
$\sum_{d\mid n}\Phi_\alpha(d)=n^\alpha$, this is exactly
$Q_b(\alpha,U)$ in \eqref{eq:Qb}.  In particular it is non-negative.
The Euler-summation remainder contributes
$O(X^{-\alpha}R_b(U))$ after absolute values are taken.  Finally,
$\zeta(\alpha)<0$ on $(0,1)$, which proves
\eqref{eq:truncated-diagonal-upper}.
\end{proof}

\begin{remark}
The secondary quadratic form is necessarily the asymmetric form
$Q_b(\alpha,U)$ in \eqref{eq:Qb}, not the symmetric expression obtained by
formally replacing $1$ by $\alpha$ in $S_b(1,U)$.  This is forced by the
factor $1/m$ in the generalised Dirichlet coefficient $a_\xi$.  The translated
source writes the secondary term in a symmetric notation; the direct
expansion above removes that ambiguity.  For the application only
$Q_b(\alpha,U)\ge0$ and $\zeta(\alpha)<0$ are needed.
\end{remark}

\begin{lemma}
\label{lem:frequency-spacing}
Let $\mathcal F(X,U)$ be the set of those frequencies $\xi<X$ that occur
in \eqref{eq:alambda-explicit}, each counted once.  Then
\[
 \#\mathcal F(X,U)\ll XZ^2,
 \qquad
 \min_{\substack{\xi,\eta\in\mathcal F(X,U)\\\xi\ne\eta}}
 |\log\xi-\log\eta|\gg\frac1{XZ^2}.
\]
Consequently, for arbitrary complex numbers $c_\xi$ and arbitrary intervals
$I_{\xi,\eta}\subseteq[T,2T]$,
\begin{equation}\label{eq:spacing-schur}
 \left|\sum_{\xi\ne\eta}c_\xi\overline{c_\eta}
 \int_{I_{\xi,\eta}}t^{\delta-1/2}
 e^{it(\log\eta-\log\xi)}\,dt\right|
 \ll T^{\delta-1/2}XZ^2\log(2XZ)
 \sum_\xi|c_\xi|^2.
\end{equation}
\end{lemma}

\begin{proof}
Write a frequency in lowest terms as $p/q$.  Since it arises as $n\ell/m$,
its denominator divides $m$, so $q\le U$; and $p<Xq$.  This proves the
cardinality bound.  For two distinct reduced fractions $p/q,r/s$,
\[
 \left|\frac pq-\frac rs\right|\ge\frac1{qs}\ge\frac1{U^2}.
\]
As both are at most $X$,
$|\log(p/q)-\log(r/s)|\ge X^{-1}|p/q-r/s|$.

Order the logarithmic frequencies increasingly and denote their minimum
spacing by $\Delta$.  Integration by parts gives
\[
 \left|\int_I t^{\delta-1/2}e^{it(u-v)}\,dt\right|
 \ll T^{\delta-1/2}|u-v|^{-1}.
\]
For a fixed frequency, the $k$th neighbour on either side is at distance at
least $k\Delta$.  Hence every row sum of the absolute kernel is
$O(T^{\delta-1/2}\Delta^{-1}\log(2\#\mathcal F))$.
The Schur test (or $2|zw|\le|z|^2+|w|^2$) proves
\eqref{eq:spacing-schur}, since $\Delta^{-1}\ll XZ^2$.
\end{proof}

\begin{proposition}
\label{prop:Zhflex}
Let the mollifier coefficients satisfy
\[
 b_1=1,\qquad |b_n|\le1\ (n\le U),\qquad b_U\ne0,
\]
where $U\asymp T^\theta$ is prime, $0<\theta<1/4$, and
\[
 \log^+\frac1{|b_U|}\ll\log(2T).
\]
In particular,
\begin{equation}\label{eq:terminal-growth-condition}
 U\log(2T)
 \left\{\log^2(2T)+\log^+\frac1{|b_U|}+1\right\}=o(T).
\end{equation}
Put
$x=1-2\delta\asymp1/\log T$.  If
\begin{equation}\label{eq:Sb-gamma}
 S_b(1,U)\le\frac{\gamma+o(1)}{\log U},
\end{equation}
then
\begin{equation}\label{eq:beta-second-moment}
 \frac1T\int_T^{2T}|\beta_U(\delta+it)|^2\,dt
 \le\frac{\gamma}{x\log U}+o(1).
\end{equation}
The horizontal argument increments required in
Lemma~\ref{lem:rectangle-sign} are $o(T)$.
\end{proposition}

\begin{proof}
Write $X_1=\sqrt{T/\pi}$, so every frequency occurring in the moving
truncation $\xi<\tau(t)$ lies below $X_1$.  The diagonal contribution of the truncated polynomial is
\begin{align*}
 D={}&\frac1T\int_T^{2T}
  \left(\frac{t}{2\pi}\right)^{\delta-1/2}
  \sum_{\xi<\sqrt{t/(2\pi)}}
  \frac{a_\xi^2}{\xi^{2\delta}}\,dt.
\end{align*}
Since $x=1-2\delta$ and
$X=\sqrt{t/(2\pi)}$, Lemma~\ref{lem:truncated-diagonal} gives
\[
 \left(\frac{t}{2\pi}\right)^{\delta-1/2}
 \sum_{\xi<X}\frac{a_\xi^2}{\xi^{2\delta}}
 \le \frac{S_b(1,U)}x
 +O\left(t^{-1/2}R_b(U)\right).
\]
Consequently
\[
 D\le\frac{S_b(1,U)}x+O(T^{-1/2}R_b(U))
 =\frac{S_b(1,U)}x+o(1).
\]
Moreover, \eqref{eq:truncated-diagonal-upper}, \eqref{eq:Sb-gamma}, and
\eqref{eq:Rb-bound} show that
\begin{equation}\label{eq:cxi-L2}
 \sum_{\xi<X_1}\frac{|a_\xi|^2}{\xi^{2\delta}}\ll1.
\end{equation}

On expanding the square of the truncated polynomial, a pair $\xi\ne\eta$
is integrated over the interval
\[
 I_{\xi,\eta}=
 [\max(T,2\pi\xi^2,2\pi\eta^2),2T],
\]
with the interval understood to be empty if its left endpoint exceeds $2T$.
The common factor $|\omega_\delta(t)|^2=1+O(T^{-1})$ has derivative
$O(T^{-2})$ and is absorbed in the same integration-by-parts estimate.  Apply
Lemma~\ref{lem:frequency-spacing} with
$c_\xi=a_\xi\xi^{-\delta}$ and $X=X_1$.  After division by $T$, the complete
off-diagonal is
\[
 \ll T^{-1}T^{\delta-1/2}X_1U^2\log(2TZ)
 \sum_{\xi<X_1}\frac{|a_\xi|^2}{\xi^{2\delta}}
 \ll T^{-1/2}U^2\log(2T)=o(1).
\]
This proves the mean square for the main truncated polynomial.

Lemma~\ref{lem:coefAFE} gives a uniform remainder
$E_U(t)\ll T^{-1/4}U\log^3T=o(1)$.  Using
$|u+v|^2\le(1+\rho)|u|^2+(1+\rho^{-1})|v|^2$ with
$\rho\downarrow0$ sufficiently slowly shows that the remainder and the cross
term contribute $o(1)$ to the normalised mean square.  This proves
\eqref{eq:beta-second-moment}.  Finally Lemma~\ref{lem:coefvariation} and
\eqref{eq:terminal-growth-condition} give \(o(T)\) for each horizontal
boundary increment; the standard bounds for \(\zeta\) and \(M_U\) on the
auxiliary rectangle are smaller.
\end{proof}

For the reciprocal-square-root coefficients,
\[
 |\lambda(n)|\le1,
 \qquad \lambda(U)=-\frac12
\]
when $U$ is prime.  Proposition~\ref{prop:Zhflex} therefore applies directly,
and Theorem~\ref{thm:transfer} supplies
$\gamma=\mathcal C_{\mathrm{res}}$.

Throughout the critical-line application we use the source normalisation
\eqref{eq:intro-c2}.  The mean-square proposition concerns the
same unscaled analytic half $\beta_U$, so no compensating factor of $1/4$
occurs in \eqref{eq:beta-second-moment}.

\section{The sharp reciprocal-square-root model}
\label{sec:sharp}

For $\Real s>1$, write
\[
 \zeta(s)^{-1/2}=\sum_{n\ge1}\frac{\lambda(n)}{n^s}.
\]
The function $\lambda$ is multiplicative and
\[
 \lambda(p^k)=(-1)^k\binom{1/2}{k}
 \quad(k\ge0),
\]
so $\lambda(p)=-1/2$, $\lambda(p^2)=-1/8$, and so forth.

Squaring the Dirichlet series in its half-plane of absolute convergence gives
\[
 \left(\sum_{n\ge1}\lambda(n)n^{-s}\right)^2
 =\zeta(s)^{-1}=\sum_{n\ge1}\mu(n)n^{-s},
\]
so that, comparing Dirichlet coefficients,
\begin{equation}\label{eq:convolution}
  (\lambda*\lambda)(n)=\mu(n)\qquad(n\ge1).
\end{equation}
This exact convolution identity is what removes the same-side poles from the
four-variable Euler product below.

Define the sharp truncated square
\[
 c_U(a):=\sum_{\substack{\ell m=a\\ \ell,m\le U}}\lambda(\ell)\lambda(m),
 \qquad a\le U^2.
\]
Then $c_U(a)=\mu(a)$ for $a\le U$, but the boundary range $U<a\le U^2$ is not
Möbius: the two individual cutoffs remain visible.

The diagonal arithmetic sum attached to this model is
\begin{equation}\label{eq:Sdef}
 \Sdiag:=\sum_{d\le U^2}\varphi(d)
 \left(\sum_{\substack{\ell,m\le U\\d\mid \ell m}}
       \frac{\lambda(\ell)\lambda(m)}{\ell m}\right)^2.
\end{equation}

\begin{lemma}\label{lem:gcd}
One has the exact identity
\begin{equation}\label{eq:gcd}
 \Sdiag=\sum_{a,b\le U^2}c_U(a)c_U(b)\frac{(a,b)}{ab}.
\end{equation}
\end{lemma}

\begin{proof}
Expanding the square in \eqref{eq:Sdef} and grouping $a=\ell_1m_1$ and
$b=\ell_2m_2$ gives
\[
 \Sdiag=\sum_{a,b\le U^2}\frac{c_U(a)c_U(b)}{ab}
       \sum_{d\mid(a,b)}\varphi(d).
\]
The classical identity $\sum_{d\mid n}\varphi(d)=n$ yields \eqref{eq:gcd}.
\end{proof}

These identities are exact.  They do not yet give the asymptotic size of
$\Sdiag$.

\section{The four-variable Euler product}
\label{sec:euler}

For $\Real x_i>0$, define
\begin{equation}\label{eq:Ddef}
 D(\mathbf x)=\sum_{\ell_1,m_1,\ell_2,m_2\ge1}
 \frac{\lambda(\ell_1)\lambda(m_1)\lambda(\ell_2)\lambda(m_2)
       (\ell_1m_1,\ell_2m_2)}
 {\ell_1^{1+x_1}m_1^{1+x_2}\ell_2^{1+x_3}m_2^{1+x_4}}.
\end{equation}
It has the Euler product $D=\prod_pL_p$, where
\begin{equation}\label{eq:Lp}
 L_p(\mathbf x)=\sum_{i,j,i',j'\ge0}
 \lambda(p^i)\lambda(p^j)\lambda(p^{i'})\lambda(p^{j'})
 p^{-\max(i+j,i'+j')}
 p^{-ix_1-jx_2-i'x_3-j'x_4}.
\end{equation}

\begin{proposition}\label{prop:factor}
There is $\eta>0$ and an Euler product $\mathcal A(\mathbf x)$, holomorphic for
$|x_i|<\eta$ in a simply connected neighbourhood reached from $\Real x_i>0$,
such that
\begin{equation}\label{eq:factor}
 D(\mathbf x)=\mathcal A(\mathbf x)
 \prod_{i=1}^4\zeta(1+x_i)^{-1/2}
 \prod_{\substack{i\in\{1,2\}\\j\in\{3,4\}}}
      \zeta(1+x_i+x_j)^{1/4}.
\end{equation}
Moreover $\mathcal A(\mathbf0)=1$.
\end{proposition}

\begin{proof}
At a fixed prime, terms of order $p^{-1}$ in \eqref{eq:Lp} arise in two ways.
If exactly one of $i,j,i',j'$ equals $1$, the contribution is
\[
 -\frac12\sum_{r=1}^4p^{-1-x_r}.
\]
If one exponent from $\{i,j\}$ and one from $\{i',j'\}$ equal $1$, the total
valuation on each side is $1$, and the contribution is
\[
 \frac14\sum_{\substack{r\in\{1,2\}\\s\in\{3,4\}}}
 p^{-1-x_r-x_s}.
\]
All other terms are $O(p^{-2+C\eta})$, uniformly for $|x_i|<\eta$.  These are
exactly the first-order terms of the local factors on the right of
\eqref{eq:factor}.  Dividing $L_p$ by those local factors therefore gives
\[
 \mathcal A_p(\mathbf x)=1+O(p^{-2+C\eta}),
\]
so $\prod_p\mathcal A_p$ converges normally after reducing $\eta$ if necessary.

At the origin, group the variables in \eqref{eq:Lp} by
$a=i+j$ and $b=i'+j'$.  Identity~\eqref{eq:convolution} gives
$\sum_{i+j=a}\lambda(p^i)\lambda(p^j)=\mu(p^a)$, hence
\[
 L_p(\mathbf0)=\sum_{a,b\in\{0,1\}}\mu(p^a)\mu(p^b)p^{-\max(a,b)}
 =1-\frac1p.
\]
The product of the eight displayed singular local factors at the origin is also
$(1-p^{-1})^2(1-p^{-1})^{-1}=1-p^{-1}$.  Thus
$\mathcal A_p(\mathbf0)=1$ for every $p$, and $\mathcal A(\mathbf0)=1$.
\end{proof}

The same-side sums $x_1+x_2$ and $x_3+x_4$ do not occur in the polar part.  At
first order they would correspond to total valuation $2$ on one side, and the
identity $\mu(p^2)=0$ removes them.  The operative singularities are therefore
the four bipartite quarter-poles.

For non-integral $U$, four applications of Perron's formula give, in the usual
truncated-integral sense,
\begin{equation}\label{eq:perron}
 \Sdiag=\frac{1}{(2\pi i)^4}\int_{(c_1)}\!\cdots\!\int_{(c_4)}
 D(\mathbf x)\frac{U^{x_1+x_2+x_3+x_4}}{x_1x_2x_3x_4}\dd\mathbf x,
 \qquad c_i>0.
\end{equation}
Formula \eqref{eq:perron} is exact after taking the truncation limits.  The
non-trivial missing step is to deform and localise these four contours with
uniform error control.

\section{The residue functional and its exact evaluation}
\label{sec:residue}

The local scaling suggested by Proposition~\ref{prop:factor} is $x_i=u_i/\log
U$.  The four Perron kernels and the eight fractional zeta factors then have net
homogeneity $(\log U)^{-1}$.  Rather than treating a conditionally convergent
four-dimensional inverse-Mellin (Bromwich) integral as a definition, we define the resulting
constant by its non-negative inverse-Laplace form.

Let $\Omega\subset[0,\infty)^4$ consist of
$\boldsymbol\tau=(\tau_{13},\tau_{14},\tau_{23},\tau_{24})$ for which
\begin{align*}
 c_1&=1-\tau_{13}-\tau_{14}>0,&
 c_2&=1-\tau_{23}-\tau_{24}>0,\\
 c_3&=1-\tau_{13}-\tau_{23}>0,&
 c_4&=1-\tau_{14}-\tau_{24}>0.
\end{align*}
Define
\begin{equation}\label{eq:Cres}
 \Cres:=\frac{1}{\Gamma(\tfrac14)^4\pi^2}
 \int_\Omega
 (\tau_{13}\tau_{14}\tau_{23}\tau_{24})^{-3/4}
 (c_1c_2c_3c_4)^{-1/2}\dd\boldsymbol\tau.
\end{equation}
The proof below both establishes convergence and evaluates the integral.

\begin{theorem}\label{thm:C}
The integral \eqref{eq:Cres} converges and
\begin{equation}\label{eq:Cvalue}
 \displaystyle
 \Cres=\frac{\Gamma(\tfrac14)^4}{2\pi^4}
       =\frac{2}{\Gamma(\tfrac34)^4}
 =0.8869411685781154\ldots.
\end{equation}
\end{theorem}

\begin{proof}
Put
\[
 B=B\!\left(\frac14,\frac12\right)
   =\frac{\Gamma(\tfrac14)^2}{\sqrt{2\pi}},
 \qquad
 F(t)={}_2F_1\!\left(\frac12,\frac14;\frac34;t\right).
\]
Set $y=1-\tau_{14}$ and $z=1-\tau_{23}$.  For fixed $y,z\in(0,1)$, the
$\tau_{13}$ and $\tau_{24}$ integrations separate and each equals
\[
 g(y,z)=\int_0^{\min(y,z)}
 t^{-3/4}\{(y-t)(z-t)\}^{-1/2}\dd t
 =m^{-1/4}M^{-1/2}B F(m/M),
\]
where $m=\min(y,z)$ and $M=\max(y,z)$.  This follows from $t=ms$ and Euler's
integral for ${}_2F_1$.  Tonelli's theorem applies because every integrand is
non-negative.  By symmetry in $y,z$,
\begin{align}
 \Gamma(\tfrac14)^4\pi^2\Cres
 &=2B^2\int_0^1(1-y)^{-3/4}y^{-1}
   \int_0^y(1-z)^{-3/4}z^{-1/2}F(z/y)^2\dd z\dd y\notag\\
 &=2B^2\int_0^1t^{-1/2}F(t)^2Y(t)\dd t,\label{eq:reduce1}
\end{align}
where, after $z=yt$,
\begin{align*}
 Y(t)&=\int_0^1y^{-1/2}(1-y)^{-3/4}(1-yt)^{-3/4}\dd y\\
 &=B\!\left(\frac12,\frac14\right)
   {}_2F_1\!\left(\frac34,\frac12;\frac34;t\right)
 =B(1-t)^{-1/2}.
\end{align*}
It remains to use the following identity, proved in
Lemma~\ref{lem:hyper}:
\begin{equation}\label{eq:K}
 \int_0^1t^{-1/2}(1-t)^{-1/2}F(t)^2\dd t=B.
\end{equation}
Equations \eqref{eq:reduce1} and \eqref{eq:K} give
\[
 \Gamma(\tfrac14)^4\pi^2\Cres=2B^4.
\]
Since $B=\Gamma(\tfrac14)^2/\sqrt{2\pi}$, this is
$\Cres=\Gamma(\tfrac14)^4/(2\pi^4)$.  The second form follows from
$\Gamma(\tfrac14)\Gamma(\tfrac34)=\pi\sqrt2$.
\end{proof}

\begin{lemma}\label{lem:hyper}
For $F(t)={}_2F_1(\tfrac12,\tfrac14;\tfrac34;t)$,
\[
 \int_0^1t^{-1/2}(1-t)^{-1/2}F(t)^2\dd t
 =B\!\left(\frac14,\frac12\right).
\]
\end{lemma}

\begin{proof}
All integrands below are non-negative.  Consequently Tonelli's theorem justifies
the interchanges, initially with values in $[0,\infty]$; the final finite value
then proves convergence throughout.  Write $B=B(\tfrac14,\tfrac12)$.  Euler's
integral gives
\[
 F(t)=B^{-1}\int_0^1s^{-3/4}(1-s)^{-1/2}(1-ts)^{-1/2}\dd s.
\]
Opening both factors of $F$ yields
\[
 B^2K=\int_0^1\!\int_0^1(s\sigma)^{-3/4}
 [(1-s)(1-\sigma)]^{-1/2}W(s,\sigma)\dd s\dd\sigma,
\]
where
\[
 W(s,\sigma)=\int_0^1
 \frac{t^{-1/2}(1-t)^{-1/2}}
      {\sqrt{(1-st)(1-\sigma t)}}\dd t.
\]
The substitution $t=u/(1-\sigma+\sigma u)$ gives
\[
 W(s,\sigma)=\pi(1-\sigma)^{-1/2}
 {}_2F_1\!\left(\frac12,\frac12;1;\frac{s-\sigma}{1-\sigma}\right).
\]
Use symmetry to restrict to $0<\sigma<s<1$ and double.  With
$\xi=(s-\sigma)/(1-\sigma)$, equivalently
$\sigma=(s-\xi)/(1-\xi)$, one obtains
\[
 B^2K=2\pi\int_0^1s^{-3/4}(1-s)^{-1/2}
 \int_0^s(s-\xi)^{-3/4}(1-\xi)^{-1/4}
 {}_2F_1\!\left(\frac12,\frac12;1;\xi\right)\dd\xi\dd s.
\]
Interchange the integrations.  The substitution
$s=\xi+(1-\xi)r$ and Euler's integral show that
\[
 \int_\xi^1s^{-3/4}(1-s)^{-1/2}(s-\xi)^{-3/4}\dd s
 =B\,\xi^{-1/2}(1-\xi)^{-1/4}.
\]
Therefore
\[
 B^2K=2\pi B\int_0^1\xi^{-1/2}(1-\xi)^{-1/2}
 {}_2F_1\!\left(\frac12,\frac12;1;\xi\right)\dd\xi.
\]
Termwise integration, justified by non-negativity, gives
\begin{align*}
 \int_0^1\xi^{-1/2}(1-\xi)^{-1/2}
 {}_2F_1\!\left(\frac12,\frac12;1;\xi\right)\dd\xi
 &=\pi\,{}_3F_2\!\left(\frac12,\frac12,\frac12;1,1;1\right)\\
 &=\frac{\pi^2}{\Gamma(\tfrac34)^4}.
\end{align*}
For the last equality, Clausen's identity gives
\[
 {}_3F_2\!\left(\frac12,\frac12,\frac12;1,1;z\right)
 ={}_2F_1\!\left(\frac14,\frac14;1;z\right)^2,
\]
and Gauss's theorem at $z=1$ gives
${}_2F_1(\tfrac14,\tfrac14;1;1)=\sqrt\pi/\Gamma(\tfrac34)^2$.
Thus
\[
 B^2K=\frac{2\pi^3B}{\Gamma(\tfrac34)^4}=B^3,
\]
where the final equality is the reflection formula
$\Gamma(\tfrac14)\Gamma(\tfrac34)=\pi\sqrt2$.  Hence $K=B$.
\end{proof}

\begin{remark}
Theorem~\ref{thm:C} evaluates the local residue functional forced by the
polar factorisation.  Theorem~\ref{thm:transfer} identifies this functional
with the sharp-cutoff limit $\lim_{U\to\infty}(\log U)\Sdiag$.
\end{remark}

\section{The sharp-box transfer theorem}
\label{sec:transfer}

The sharp-box transfer is proved by separating the signed part of the Euler
product into four one-variable Selberg--Delange tails.  The four cross factors
have non-negative coefficients, which permits a direct logarithmic-measure
argument.

For $z\in\mathbb C$, let $\tau_z(n)$ be defined by
\[
  \zeta(s)^z=\sum_{n\ge1}\frac{\tau_z(n)}{n^s}
  \qquad (\Real s>1),
\]
with the branch that is positive for real $s>1$.  Thus
$\tau_{-1/2}=\lambda$ and $\tau_{1/4}(n)\ge0$.  Put
\[
 \alpha=\frac14,\qquad \beta=\frac12,
 \qquad
 L(X)=\sum_{n\le X}\frac{\lambda(n)}n,
\]
where $L(X)=0$ for $0<X<1$.

\begin{lemma}
\label{lem:oneSD}
As $t\to\infty$,
\begin{align}
 \sum_{n\le e^t}\frac{\tau_{1/4}(n)}n
   &=\frac{t^{1/4}}{\Gamma(5/4)}+O(t^{-3/4}),
   \label{eq:SDplus}\\
 L(e^t)&=\frac{t^{-1/2}}{\Gamma(1/2)}+O(t^{-3/2}).
 \label{eq:SDminus}
\end{align}
Moreover, for $t\ge0$,
\begin{equation}
 |L(e^t)|\ll (1+t)^{-1/2}.
 \label{eq:Lbound}
\end{equation}
\end{lemma}

\begin{proof}
This is the classical Selberg--Delange theorem applied, after Perron inversion,
to $\zeta(1+s)^z$ with $z=1/4$ and $z=-1/2$.  For completeness, near $s=0$ one
has
\[
 \zeta(1+s)^z=s^{-z}\{1+O_z(s)\}.
\]
After moving the Perron contour to a Hankel contour around the negative real
axis in the standard zero-free region, the leading inverse-Laplace integral is
\[
 \frac{1}{2\pi i}\int_{\mathcal H}e^{ts}s^{-z-1}\dd s
 =\frac{t^z}{\Gamma(z+1)}.
\]
The next term in the local expansion is $O(t^{\Real z-1})$, while the remaining
contour is smaller than every fixed negative power of $t$.  This gives
\eqref{eq:SDplus} and \eqref{eq:SDminus}.  Equation \eqref{eq:Lbound} follows
from \eqref{eq:SDminus} for $t\ge2$ and by enlarging the implied constant on the
compact interval $0\le t\le2$.  More precisely, apply
\cite[Theorem~1.1, pp.~350--351]{deLaBretecheTenenbaum} with
\(f=\tau_z\), \(\varrho=z\), and regular factor \(G=1\), first for
\(z=1/4\) and then for \(z=-1/2\), and use partial summation.  The branch in
that theorem is the one fixed above; its arbitrary parameter \(A\) supplies
the two displayed remainders.
\end{proof}

Let $R\ge2$ and introduce the positive measure
\begin{equation}
 d\nu_R(u)=R^{-1/4}\sum_{n\ge1}\frac{\tau_{1/4}(n)}n
 \delta_{\log n/R}(u),
 \label{eq:nuR}
\end{equation}
and the rescaled signed tail
\begin{equation}
 W_R(v)=R^{1/2}L(e^{Rv}),\qquad v\ge0.
 \label{eq:WR}
\end{equation}

Here and below, $\mu_R\Longrightarrow\mu$ denotes weak convergence of
measures, that is, $\int f\dd\mu_R\to\int f\dd\mu$ for every bounded
continuous $f$ supported in the interval concerned.

\begin{lemma}
\label{lem:logmeasure}
On every fixed compact interval of $[0,\infty)$,
\[
 \nu_R\Longrightarrow \frac{u^{-3/4}}{\Gamma(1/4)}\dd u.
\]
Also, locally uniformly for $v>0$,
\[
 W_R(v)\longrightarrow \frac{v^{-1/2}}{\sqrt\pi},
 \qquad
 |W_R(v)|\ll (v+R^{-1})^{-1/2}.
\]
Finally, if
$I_k=[k/R,(k+1)/R)$, then uniformly for $k\ge0$,
\begin{equation}
 \nu_R(I_k)\ll
 R^{-1}\left(\frac{k+1}{R}\right)^{-3/4}.
 \label{eq:cellmass}
\end{equation}
\end{lemma}

\begin{proof}
The convergence of distribution functions follows from \eqref{eq:SDplus}:
for fixed $u>0$,
\[
 \nu_R([0,u])
 =R^{-1/4}\sum_{n\le e^{Ru}}\frac{\tau_{1/4}(n)}n
 =\frac{u^{1/4}}{\Gamma(5/4)}+O(R^{-1}u^{-3/4}).
\]
The assertions about $W_R$ follow from \eqref{eq:SDminus} and
\eqref{eq:Lbound}.  For $k\ge1$, subtracting \eqref{eq:SDplus} at $t=k+1$ and
$t=k$ gives
\[
 \sum_{e^k<n\le e^{k+1}}\frac{\tau_{1/4}(n)}n\ll k^{-3/4}.
\]
After multiplication by $R^{-1/4}$ this is \eqref{eq:cellmass}; the case $k=0$
is immediate.
\end{proof}

Let
\[
 E=\{13,14,23,24\},
\]
and, for $\boldsymbol q=(q_e)_{e\in E}$, put
\[
 Q_1=q_{13}q_{14},\quad Q_2=q_{23}q_{24},\quad
 Q_3=q_{13}q_{23},\quad Q_4=q_{14}q_{24}.
\]
Define the polar model
\begin{equation}
 D_0(\mathbf x)=
 \prod_{i=1}^4\zeta(1+x_i)^{-1/2}
 \prod_{\substack{i\in\{1,2\}\\j\in\{3,4\}}}
 \zeta(1+x_i+x_j)^{1/4}.
 \label{eq:D0}
\end{equation}
If $S_0(Y_1,Y_2,Y_3,Y_4)$ denotes the sharp box sum of the Dirichlet
coefficients of $D_0$, then expanding first the four cross factors gives the
exact identity
\begin{equation}
 S_0(\mathbf Y)=
 \sum_{\substack{q_e\ge1\\Q_i\le Y_i\ (1\le i\le4)}}
 \prod_{e\in E}\frac{\tau_{1/4}(q_e)}{q_e}
 \prod_{i=1}^4 L(Y_i/Q_i).
 \label{eq:S0conv}
\end{equation}

For $\mathbf b=(b_1,b_2,b_3,b_4)$ with positive coordinates, let
$\mathcal P(\mathbf b)$ be the polytope of $\boldsymbol u=(u_e)_{e\in E}\ge0$
for which
\begin{align*}
 c_1&=b_1-u_{13}-u_{14}\ge0,&
 c_2&=b_2-u_{23}-u_{24}\ge0,\\
 c_3&=b_3-u_{13}-u_{23}\ge0,&
 c_4&=b_4-u_{14}-u_{24}\ge0.
\end{align*}
Set
\begin{equation}
 \mathfrak C(\mathbf b)=
 \frac{1}{\Gamma(1/4)^4\pi^2}
 \int_{\mathcal P(\mathbf b)}
 \prod_{e\in E}u_e^{-3/4}
 \prod_{i=1}^4c_i^{-1/2}\dd\boldsymbol u.
 \label{eq:Cb}
\end{equation}
\begin{lemma}
\label{lem:polyhedral-ui}
The integral in \eqref{eq:Cb} is locally finite.  Moreover, on every compact
neighbourhood of $\mathbf1$ its integrand has a common integrable majorant
after the natural translation of the moving faces; consequently
$\mathfrak C(\mathbf b)$ is continuous there.
\end{lemma}

\begin{proof}
It is enough first to treat $\mathbf b=\mathbf1$.  Put
\[
 g(y,z)=\int_0^{\min(y,z)}
 t^{-3/4}(y-t)^{-1/2}(z-t)^{-1/2}\dd t.
\]
Integrating successively in $u_{13}$ and $u_{24}$ reduces the four-dimensional
integral, up to its fixed normalising factor, to
\[
 \int_0^1\!\int_0^1
 u^{-3/4}v^{-3/4}g(1-u,1-v)^2\,du\,dv.
\]
If $m=\min(y,z)$ and $M=\max(y,z)$, Euler's integral gives
\[
 g(y,z)=B(\tfrac14,\tfrac12)m^{-1/4}M^{-1/2}
 {}_2F_1(\tfrac12,\tfrac14;\tfrac34;m/M).
\]
The hypergeometric function is zero-balanced, and hence
\[
 g(y,z)\ll m^{-1/4}M^{-1/2}
 \left(1+\log^+\frac{M}{|M-m|}\right).
\]
Consider, by symmetry, the region $0<y\le z<1$.  Near $y=z=0$ the resulting
majorant is
\[
 y^{-1/2}z^{-1}
 \left(1+\log\frac{z}{z-y}\right)^2,
\]
whose $y$-integral is $O(z^{1/2})$ after the substitution $y=zr$; the
remaining $z^{-1/2}$ singularity is integrable.  Near $y=z=1$, the only power
singularities are $(1-y)^{-3/4}(1-z)^{-3/4}$, and the logarithmic singularity
along $y=z$ is locally integrable.  The mixed boundary regions are easier.
This proves local finiteness.

If $\mathbf b$ ranges in a sufficiently small compact neighbourhood of
$\mathbf1$, translate each moving face by the corresponding change in
$b_i$.  The preceding estimates remain valid with uniform constants and on a
fixed slightly enlarged compact domain.  They therefore provide a common
integrable majorant.  Dominated convergence proves the continuity of
$\mathfrak C(\mathbf b)$.
\end{proof}

In particular,
\begin{equation}
 \mathfrak C(1,1,1,1)=\Cres.
 \label{eq:CbCres}
\end{equation}

\begin{proposition}
\label{prop:polartransfer}
Uniformly for $\mathbf b$ in a sufficiently small fixed compact neighbourhood
of $(1,1,1,1)$,
\begin{equation}
 S_0(e^{Rb_1},e^{Rb_2},e^{Rb_3},e^{Rb_4})
 =\frac{\mathfrak C(\mathbf b)}R+o(R^{-1}).
 \label{eq:polartransfer}
\end{equation}
\end{proposition}

\begin{proof}
From \eqref{eq:S0conv}, \eqref{eq:nuR}, and \eqref{eq:WR}, one obtains the
exact rescaled identity
\begin{equation}
 R S_0(e^{Rb_1},e^{Rb_2},e^{Rb_3},e^{Rb_4})
 =\int_{\mathcal P(\mathbf b)}
   \prod_{i=1}^4W_R(c_i)\dd\nu_R^{\otimes4}(\boldsymbol u).
 \label{eq:exactscaled}
\end{equation}
We justify passage to the limit, including all boundary faces.

Fix $\varepsilon>0$ and first restrict to the compact subset on which
$u_e\ge\varepsilon$ for every $e$ and $c_i\ge\varepsilon$ for every $i$.
Lemma~\ref{lem:logmeasure}, weak convergence of product measures, and local
uniform convergence of $W_R$ give the required limit there, uniformly for
$\mathbf b$ in the stated compact set.

It remains to prove uniform integrability.  Partition each $u_e$-axis into the
intervals $I_k$ of Lemma~\ref{lem:logmeasure}.  By \eqref{eq:cellmass}, the
$\nu_R^{\otimes4}$-mass of a mesh box $B=\prod_e I_{k_e}$ is at most a constant
times
\[
 R^{-4}\prod_{e\in E}(u_{e,B}+R^{-1})^{-3/4},
\]
where $u_{e,B}=k_e/R$.  On a box meeting $\mathcal P(\mathbf b)$,
Lemma~\ref{lem:logmeasure} bounds the kernel by a constant times
\[
 \prod_{i=1}^4(c_{i,B}^{-}+R^{-1})^{-1/2},
 \qquad
 c_{i,B}^{-}=
 \max\!\left(0,b_i-\sum_{e\ni i}\frac{k_e+1}{R}\right).
\]
Consequently the contribution of any collection $\mathcal B$ of such mesh
boxes is bounded by a constant times the corresponding singular Riemann sum.
To compare this sum with an integral, take in every mesh box a fixed sub-box of
side $1/(8R)$ adjacent to its upper corner.  On that sub-box all coordinate
factors and all regularised deficit factors are comparable, with constants
independent of the box, $R$, and $\mathbf b$.  The sub-boxes are disjoint.  It
follows that
\begin{align}
 &\int_{\bigcup\mathcal B\cap\mathcal P(\mathbf b)}
  \prod_i|W_R(c_i)|\dd\nu_R^{\otimes4}
 \notag\\
 &\quad\ll
 \int_{\mathcal N_{3/R}(\bigcup\mathcal B)\cap
             \mathcal P(\mathbf b+3R^{-1}\mathbf1)}
 \prod_{e\in E}u_e^{-3/4}
 \prod_{i=1}^4
 \left(b_i+3R^{-1}-\sum_{e\ni i}u_e\right)^{-1/2}
 \dd\boldsymbol u.
 \label{eq:gridcompare}
\end{align}
Here $\mathcal N_{3/R}$ denotes a $3/R$-neighbourhood.  The enlargement by
$3/R$ merely ensures that every regularised mesh factor is dominated by the
corresponding unregularised factor on the enlarged polytope.

The right-hand side of \eqref{eq:gridcompare} is uniformly bounded by
Lemma~\ref{lem:polyhedral-ui}.  More importantly, if $\mathcal B$ consists of boxes meeting one
of the sets $u_e<\varepsilon$ or $c_i<\varepsilon$, then the right-hand side
tends to zero uniformly as $\varepsilon\downarrow0$, after first taking
$R\to\infty$.  This is simply the absolute continuity of the finite integrals
\eqref{eq:Cb}; the $3/R$ enlargement disappears in the limit.  Thus the family
of integrands in \eqref{eq:exactscaled} is uniformly integrable.

Combining the interior convergence with this boundary estimate gives
\[
 R S_0(e^{R\mathbf b})\longrightarrow\mathfrak C(\mathbf b)
\]
uniformly in the stated neighbourhood, proving \eqref{eq:polartransfer}.
\end{proof}

We next restore the regular Euler product $\mathcal A$ from
Proposition~\ref{prop:factor}.

\begin{lemma}
\label{lem:Acoeff}
There is $\eta>0$ and coefficients $a(\mathbf r)$, indexed by
$\mathbf r=(r_1,r_2,r_3,r_4)\in\mathbb N^4$, such that
\begin{equation}
 \mathcal A(\mathbf x)=
 \sum_{\mathbf r\in\mathbb N^4}
 \frac{a(\mathbf r)}{r_1^{x_1}r_2^{x_2}r_3^{x_3}r_4^{x_4}},
 \qquad
 \sum_{\mathbf r}|a(\mathbf r)|(r_1r_2r_3r_4)^\eta<\infty.
 \label{eq:Aseries}
\end{equation}
Moreover,
\begin{equation}
 \sum_{\mathbf r}a(\mathbf r)=\mathcal A(\mathbf0)=1.
 \label{eq:Asum}
\end{equation}
\end{lemma}

\begin{proof}
The proof of Proposition~\ref{prop:factor} gives, after the first-order local
terms have been divided out,
\[
 \mathcal A_p(\mathbf x)=1+O(p^{-2+C\eta})
\]
uniformly for $|\Real x_i|\le\eta$, once $\eta$ is small enough.  Expanding the
local binomial factors and the absolutely convergent local sum \eqref{eq:Lp}
shows more precisely that the sum of the absolute values of the non-constant
local Dirichlet coefficients, weighted by $p^{\eta(\nu_1+\cdots+\nu_4)}$, is
$O(p^{-2+C\eta})$.  The product over $p$ therefore converges in the weighted
$\ell^1$ coefficient norm.  This gives \eqref{eq:Aseries}.  Absolute convergence
at $\mathbf x=\mathbf0$ and Proposition~\ref{prop:factor} give
\eqref{eq:Asum}.
\end{proof}

\begin{theorem}
\label{thm:transfer}
As $U\to\infty$,
\begin{equation}
 \displaystyle
  \Sdiag=\frac{\Cres}{\log U}
  +o\!\left(\frac1{\log U}\right).
 \label{eq:transferproved}
\end{equation}
\end{theorem}

\begin{proof}
Write $R=\log U$.  Let $b_0(\mathbf n)$ denote the Dirichlet coefficients of
$D_0$ in the variables $\mathbf x$, so that $S_0(\mathbf Y)$ is their box sum.
Equations \eqref{eq:factor} and \eqref{eq:Aseries} imply the exact coordinatewise
convolution identity
\begin{equation}
 \Sdiag=
 \sum_{\mathbf r}a(\mathbf r)
 S_0(U/r_1,U/r_2,U/r_3,U/r_4),
 \label{eq:Aconvolution}
\end{equation}
with the convention that $S_0(\mathbf Y)=0$ if some $Y_i<1$.

We first record a crude bound.  The Euler products at
$1+1/R$ give
\[
 \sum_{n\le U}\frac{|\lambda(n)|}{n}\ll R^{1/2},
 \qquad
 \sum_{n\le U}\frac{\tau_{1/4}(n)}n\ll R^{1/4}.
\]
Indeed, the first absolute-value Dirichlet series has first prime coefficient
$1/2$ and hence is $\ll\zeta(1+1/R)^{1/2}$, while the second is exactly
$\zeta(1+1/R)^{1/4}$; Rankin's trick removes the sharp cutoff.  From
\eqref{eq:S0conv}, uniformly for $1\le Y_i\le U$,
\begin{equation}
 |S_0(\mathbf Y)|\ll R^3.
 \label{eq:S0crude}
\end{equation}

Split \eqref{eq:Aconvolution} at
$\max_i r_i\le e^{\sqrt R}$.  By \eqref{eq:Aseries} and \eqref{eq:S0crude}, the
tail is
\[
 \ll R^3e^{-\eta\sqrt R}
 \sum_{\mathbf r}|a(\mathbf r)|(r_1r_2r_3r_4)^\eta
 =o(R^{-1}).
\]
In the remaining range put
\[
 b_i=1-\frac{\log r_i}{R}.
\]
Then $1-R^{-1/2}\le b_i\le1$.  Proposition~\ref{prop:polartransfer}, uniformly
in this range, and the continuity of $\mathfrak C$ at $\mathbf1$ give
\[
 R S_0(U/r_1,U/r_2,U/r_3,U/r_4)=\Cres+o(1),
\]
where the $o(1)$ is uniform in $\mathbf r$.  Therefore
\begin{align*}
 R\Sdiag
 &=\{\Cres+o(1)\}
   \!\!\sum_{\max r_i\le e^{\sqrt R}}\!\!\!\!a(\mathbf r)+o(1)\\
 &=\Cres+o(1)
\end{align*}
by \eqref{eq:Asum} and weighted absolute convergence.  This is
\eqref{eq:transferproved}.
\end{proof}

\begin{remark}
The proof does not require a multivariable Tauberian theorem for a signed
arithmetic function.  The four signed factors are treated separately by the
one-variable estimate \eqref{eq:SDminus}; after that separation, the four edge
measures in \eqref{eq:nuR} are positive.  This is why the non-negativity
restriction in the multivariable theorem of de la Bret\`eche~\cite{deLaBreteche}
is avoided.
\end{remark}

\section{Weighted power transfer}
\label{sec:weights}

For $a>0$ put
\[
 b_{n,a}(U)=\lambda(n)
 \left(\frac{\log(U/n)}{\log U}\right)^a\quad(n\le U)
\]
and let $S_a(U)$ denote \eqref{eq:Sb} for these coefficients.  Write
$R=\log U$ and
\[
 L_{a,R}(c)=\sum_{n\le e^{Rc}}\frac{\lambda(n)}n
 \left(c-\frac{\log n}{R}\right)^a,
 \qquad W_{a,R}(c)=R^{1/2}L_{a,R}(c).
\]
If $L(X)=\sum_{n\le X}\lambda(n)/n$, then Stieltjes summation gives
\begin{equation}\label{eq:weightedStieltjes}
 L_{a,R}(c)=a\int_0^cL(e^{R(c-v)})v^{a-1}\dd v.
\end{equation}
The one-variable estimate \eqref{eq:SDminus} and
\eqref{eq:Lbound} imply, locally uniformly for $c>0$,
\begin{equation}\label{eq:weightedlimit}
 W_{a,R}(c)\longrightarrow
 K_ac^{a-1/2},\qquad
 K_a=\frac{\Gamma(a+1)}{\Gamma(a+1/2)},
\end{equation}
and
\begin{equation}\label{eq:weightedbound}
 |W_{a,R}(c)|\ll_a(c+R^{-1})^{a-1/2}.
\end{equation}

On the polytope $\Omega$ of Section~\ref{sec:residue}, define
\begin{equation}\label{eq:Capower}
 \mathcal C_a=
 \frac{K_a^4}{\Gamma(1/4)^4}
 \int_\Omega
 (\tau_{13}\tau_{14}\tau_{23}\tau_{24})^{-3/4}
 (c_1c_2c_3c_4)^{a-1/2}\dd\boldsymbol\tau.
\end{equation}

\begin{lemma}
\label{lem:weighted-polar-uniform}
Let $\mathbf b=(b_1,b_2,b_3,b_4)$ lie in a fixed neighbourhood of
$\mathbf1$, and impose the coordinate caps $n_i\le e^{Rb_i}$ in the polar
model $D_0$.  Weight the $i$th coordinate by
\[
 \left(b_i-\frac{\log n_i}{R}\right)^a.
\]
Denote the resulting polar box sum by $S_{a,0,R}(\mathbf b)$.  Then, uniformly
for $\mathbf b\to\mathbf1$,
\begin{equation}\label{eq:weighted-polar-uniform}
 R S_{a,0,R}(\mathbf b)
 \longrightarrow
 \mathfrak C_a(\mathbf b),
\end{equation}
where
\[
 \mathfrak C_a(\mathbf b)=
 \frac{K_a^4}{\Gamma(1/4)^4}
 \int_{\mathcal P(\mathbf b)}
 \prod_{e\in E}u_e^{-3/4}
 \prod_{i=1}^4
 \left(b_i-\sum_{e\ni i}u_e\right)^{a-1/2}
 \dd\boldsymbol u.
\]
In particular $\mathfrak C_a(\mathbf1)=\mathcal C_a$.
\end{lemma}

\begin{proof}
After expansion of the four positive cross factors, the edge variables are
again governed by the positive measures $\nu_R$ from
Proposition~\ref{prop:polartransfer}; the four coordinate tails are
$W_{a,R}(c_i)$.  On every compact subset of the interior,
\eqref{eq:weightedlimit} gives uniform convergence to
$K_ac_i^{a-1/2}$.  Equation \eqref{eq:weightedbound} gives the mesh majorant
\[
 \prod_{e\in E}(u_e+R^{-1})^{-3/4}
 \prod_{i=1}^4(c_i+R^{-1})^{a-1/2}.
\]
If $0<a<1/2$, this is bounded by a constant times the sharp majorant with
$c_i^{-1/2}$; if $a\ge1/2$, the deficit factors are locally bounded.  The
sub-box comparison used in \eqref{eq:gridcompare} therefore applies uniformly
for $\mathbf b$ near $\mathbf1$.  Boundary boxes meeting $u_e<\varepsilon$ or
$c_i<\varepsilon$ contribute $o_\varepsilon(1)$ uniformly after
$R\to\infty$, while the interior boxes converge as ordinary Riemann sums.
Letting $\varepsilon\downarrow0$ proves \eqref{eq:weighted-polar-uniform}.
Local dominated convergence also shows that $\mathfrak C_a(\mathbf b)$ is
continuous at $\mathbf1$.
\end{proof}

\begin{lemma}
\label{lem:weighted-A-convolution}
Let $a(\mathbf r)$ be the coefficients in Lemma~\ref{lem:Acoeff}.  Then
\begin{equation}\label{eq:weighted-A-exact}
 S_a(e^R)=
 \sum_{\mathbf r}a(\mathbf r)
 S_{a,0,R}\!\left(
 1-\frac{\log r_1}{R},\ldots,
 1-\frac{\log r_4}{R}
 \right),
\end{equation}
with the convention that a term is zero if some $r_i>e^R$.  Moreover,
\begin{equation}\label{eq:weighted-A-limit}
 R S_a(e^R)=\mathcal C_a+o(1).
\end{equation}
\end{lemma}

\begin{proof}
The Dirichlet-series identity $D=\mathcal A D_0$ gives coordinatewise
convolution.  If the regular factor contributes $r_i$ to the $i$th coordinate,
then
\[
 \left(1-\frac{\log(r_in_i)}R\right)^a
 =\left(b_i-\frac{\log n_i}R\right)^a,
 \qquad b_i=1-\frac{\log r_i}R,
\]
which proves the exact identity \eqref{eq:weighted-A-exact}.

Since the weights lie in $[0,1]$, the absolute polar sum satisfies the same
crude estimate as \eqref{eq:S0crude}:
\begin{equation}\label{eq:weighted-polar-crude}
 |S_{a,0,R}(\mathbf b)|\ll R^3
 \qquad(0\le b_i\le1).
\end{equation}
Split \eqref{eq:weighted-A-exact} at
$\max_i r_i\le e^{\sqrt R}$.  Weighted absolute convergence from
\eqref{eq:Aseries} and \eqref{eq:weighted-polar-crude} give for the tail
\[
 \ll R^3e^{-\eta\sqrt R}
 \sum_{\mathbf r}|a(\mathbf r)|(r_1r_2r_3r_4)^\eta
 =o(R^{-1}).
\]
In the remaining range,
$1-R^{-1/2}\le b_i\le1$.  Lemma~\ref{lem:weighted-polar-uniform} and continuity at $\mathbf1$ give, uniformly
in $\mathbf r$,
\[
 R S_{a,0,R}(\mathbf b)=\mathcal C_a+o(1).
\]
Therefore
\[
 R S_a(e^R)
 =\{\mathcal C_a+o(1)\}
 \sum_{\max r_i\le e^{\sqrt R}}a(\mathbf r)+o(1)
 =\mathcal C_a+o(1),
\]
using \eqref{eq:Asum}.  This proves \eqref{eq:weighted-A-limit}.
\end{proof}

\begin{theorem}\label{thm:powertransfer}
For every fixed $a>0$,
\[
 \displaystyle
 S_a(U)=\frac{\mathcal C_a}{\log U}
 +o\!\left(\frac1{\log U}\right).
\]
\end{theorem}

\begin{proof}
Take $R=\log U$ in Lemma~\ref{lem:weighted-A-convolution}.
\end{proof}

\subsection{Zhuravlev's coefficients evaluated exactly}
\label{sec:Zh-exact}

The transfer theorems above use $\lambda$ only through
Proposition~\ref{prop:factor} and through $\sum_{\mathbf r}a(\mathbf
r)=\mathcal A(\mathbf0)$.  Both survive for a general squarefree-supported
multiplicative $b$ with $b(p)=-v_p$ and $v_p\to\tfrac12$, because the
first-order terms of the local factor are unchanged: exactly one exponent equal
to $1$ contributes $-v_p\sum_rp^{-1-x_r}$, one exponent from each side
contributes $v_p^2\sum_{r,s}p^{-1-x_r-x_s}$, and everything else is
$O(p^{-2})$, the same-side terms now having total valuation two rather than
vanishing coefficient.  Writing $w_0=1$, $w_1=-2v_p$, $w_2=v_p^2$ for the inner
sums $\sum_{i+j=k}b(p^i)b(p^j)$ and using
$L_p(\mathbf0)=\sum_{k,l\le2}w_kw_lp^{-\max(k,l)}$, one gets
\begin{equation}\label{eq:general-A}
 \mathcal A_b(\mathbf 0)=\prod_p\frac{L_p(\mathbf0)}{1-p^{-1}},
 \qquad
 L_p(\mathbf0)=1+\frac{4v_p(v_p-1)}{p}
   +\frac{v_p^2(2-4v_p+v_p^2)}{p^2},
\end{equation}
and, for coefficients carrying the profile of exponent $a$,
\begin{equation}\label{eq:general-transfer}
 S_b(1,U)=\frac{\mathcal A_b(\mathbf0)\,\mathcal C_a}{\log U}
 +o\!\left(\frac1{\log U}\right).
\end{equation}

Three consequences bear on Section~\ref{sec:Zhcoeff}.

\emph{The normalisation $p f(p)\to-\tfrac12$ is forced, not chosen.}  For
constant $v$ the polar exponents in \eqref{eq:factor} become $-v$ and $v^2$,
the net homogeneity of the Perron integral is $(\log U)^{4v(v-1)}$, and
\[
 4v(v-1)=4\bigl(v-\tfrac12\bigr)^2-1\ge-1,
\]
with equality only at $v=\tfrac12$.  Any other limiting value costs a power of
$\log U$, not merely a constant.  Zhuravlev's $f(p)\sim-1/(2p)$ is therefore
the unique admissible normalisation, and it is exactly $\lambda(p)=-\tfrac12$
recovered on squarefree support.

\emph{Squarefree support is cheap.}  For $b=\mu^2\lambda$ one has
$v_p\equiv\tfrac12$, so $L_p(\mathbf0)=1-p^{-1}+\tfrac1{16}p^{-2}$ and
\begin{equation}\label{eq:sqfree-A}
 \mathcal A_{\mu^2\lambda}(\mathbf0)
 =\prod_p\left(1+\frac1{16p(p-1)}\right)
 =1.0489438\ldots,
\end{equation}
so discarding the prime powers of $\lambda$ costs under $5\%$.

\emph{Zhuravlev's normalisation supplies a profile.}  Since
$b_\ell=\ell f(\ell)N(U/\ell;\ell)/N(U)$ with
$N(x;\ell)=\sum_{q\le x,(q,\ell)=1}|f(q)|$, and since $|f|$ has the Dirichlet
series $\prod_p(1+|f(p)|p^{-s})=\zeta(1+s)^{1/2}\times(\text{regular})$, one has
$N(x)\asymp(\log x)^{1/2}$ and hence
\[
 \frac{N(U/\ell)}{N(U)}\sim\left(\frac{\log(U/\ell)}{\log U}\right)^{1/2}.
\]
His coefficients are thus, up to a convergent Euler product, the power profile
of Section~\ref{sec:weights} with $a=\tfrac12$.  At that exponent the factors
$c_i^{a-1/2}$ in \eqref{eq:Capower} disappear and $\mathcal C_{1/2}$ is
elementary: for fixed $\tau_{14},\tau_{23}$ the remaining two variables are
each free up to $m=1-\max(\tau_{14},\tau_{23})$, so the inner integrals factor
as $(4m^{1/4})^2$ and
\[
 \int_\Omega(\tau_{13}\tau_{14}\tau_{23}\tau_{24})^{-3/4}\dd\boldsymbol\tau
 =128\int_0^1y^{-1/2}(1-y)^{1/2}\dd y
 =128\,B\!\left(\tfrac12,\tfrac32\right)=64\pi .
\]
With $K_{1/2}=\sqrt\pi/2$ this gives the closed form
\begin{equation}\label{eq:Chalf}
 \mathcal C_{1/2}=\frac{4\pi^3}{\Gamma(\tfrac14)^4}
 =\frac{\Gamma(\tfrac34)^4}{\pi}
 =0.7177700110\ldots,
\end{equation}
against $\mathcal C_0=\Cres=0.8869411\ldots$; the profile is worth a factor
$\Gamma(\tfrac34)^8/(2\pi)=0.8093$.  As a check, \eqref{eq:Chalf} also follows
from the one-dimensional reduction \eqref{eq:Careduction} at $a=\tfrac12$, where
the first hypergeometric factor degenerates to $1$; the two agree to twelve
decimal places.

\begin{remark}\label{rem:Zh-exact}
Combining \eqref{eq:general-A}, \eqref{eq:general-transfer} and
\eqref{eq:Chalf} at Zhuravlev's own $v_p=p|f(p)|$ gives
$\mathcal A(\mathbf0)=1.1218\ldots$, so that his coefficients satisfy
\[
 S_b(1,U)=\frac{0.8052\ldots}{\log U}+o\!\left(\frac1{\log U}\right),
 \qquad\text{whence}\qquad
 \kappa_0\ge0.0571\ldots.
\]
These are of course not the numbers Zhuravlev states.  His theorem gives
$\kappa_0\ge0.0248493\ldots$ in the corrected form of
Corollary~\ref{cor:Zh-corrected}, and the two differ for a reason that has
nothing to do with his choice of $f$: he estimates $S_b(1,U)$ by Wirsing's
mean-value theorem, which was the natural tool for a sum of this shape in 1974,
whereas the calculation above evaluates the same sum exactly, using the
four-variable Euler product of Section~\ref{sec:euler}.  Passing from
$S_b(1,U)$ to $\kappa_0$ also uses Section~\ref{sec:corrections} and the
coefficient-uniform estimates of Section~\ref{sec:Zsource}.  What the
calculation shows is that his coefficients were well chosen.  Wirsing's
inequality accounts for a factor $2.3$ on its own; optimising $v_p$ prime by
prime inside \eqref{eq:general-A} moves the constant by under one per cent, and
$f(2),f(3),f(5),f(7)$ are already within a few per cent of the minimisers.  By
the local inequality of Section~\ref{sec:psd} the multiplicative structure
offers nothing further either.  The freedom that remains is the profile, and
that is what Sections~\ref{sec:weights} and~\ref{sec:psd} exploit.
\end{remark}

It is natural to ask whether one gains by keeping Zhuravlev's coefficients
in place of $\lambda$.  Ordering the constants of
\eqref{eq:general-transfer}, smallest first,
\[
\begin{array}{lcc}
\lambda,\ 0.966v^{1/4}+0.034v\ \text{(Theorem~\ref{thm:profile})} & 0.6781 & \kappa_0>0.0678190\\
\lambda,\ a=\tfrac14 & 0.6783 & \kappa_0>0.0677586\\
\lambda,\ a=\tfrac12 & 0.7178 & \kappa_0>0.0640660\\
\text{Zhuravlev's }f & 0.8052 & \kappa_0>0.0571\\
\lambda,\ \text{sharp cutoff} & 0.8869 & \kappa_0>0.0518466\\
\text{Wirsing's bound on Zhuravlev's }f & 1.8506 & \kappa_0>0.0248493
\end{array}
\]
One does not: his choice sits inside the family treated here, with profile exponent $\tfrac12$ in place of the more favourable $\tfrac14$, and with the arithmetic factor $\mathcal A(\mathbf0)>1$ of \eqref{eq:sqfree-A} attached to squarefree support, which is independent of the profile and so cannot be traded against it.  Relaxing either brings one back to $\lambda$ at $a=\tfrac14$, and the profile optimisation of Section~\ref{sec:psd} improves slightly on that in turn.  The 1974 choice comes close, and already lies above the sharp-cutoff model of Section~\ref{sec:sharp}.

\subsection{The quarter-power constant}

The four-dimensional integral admits the following one-dimensional reduction.  Put
$\alpha=1/4$ and $q=a+1/2$.  After setting
$y=1-\tau_{14}$ and $z=1-\tau_{23}$, the integrations in
$\tau_{13}$ and $\tau_{24}$ each produce
\[
 H_a(y,z):=\int_0^{\min(y,z)}
 t^{\alpha-1}(y-t)^{q-1}(z-t)^{q-1}\dd t.
\]
If $m=\min(y,z)$ and $M=\max(y,z)$, Euler's beta integral gives
\begin{equation}\label{eq:Ha}
 H_a(y,z)=B(\alpha,q)m^{\alpha+q-1}M^{q-1}
 {}_2F_1(1-q,\alpha;\alpha+q;m/M).
\end{equation}
Using the symmetry in $y,z$, restricting to $0<z<y<1$, and writing $z=yt$,
we therefore obtain
\begin{align}
 &\int_\Omega
 (\tau_{13}\tau_{14}\tau_{23}\tau_{24})^{-3/4}
 (c_1c_2c_3c_4)^{a-1/2}\dd\boldsymbol\tau \notag\\
 &\quad=2B\left(\frac14,a+\frac12\right)^2
 B\left(4a+\frac12,\frac14\right)
 \int_0^1 t^{2a-1/2}
 {}_2F_1\left(\frac12-a,\frac14;a+\frac34;t\right)^2 \notag\\
 &\hspace{42mm}\times
 {}_2F_1\left(\frac34,4a+\frac12;4a+\frac34;t\right)\dd t.
 \label{eq:Careduction}
\end{align}
Indeed, after \eqref{eq:Ha} the remaining $y$-integral is
\[
 \int_0^1y^{4a-1/2}(1-y)^{-3/4}(1-yt)^{-3/4}\dd y
 =B\left(4a+\frac12,\frac14\right)
 {}_2F_1\left(\frac34,4a+\frac12;4a+\frac34;t\right).
\]

For $a=1/4$, Euler's transformation gives
\[
 {}_2F_1\left(\frac34,\frac32;\frac74;t\right)
 =(1-t)^{-1/2}{}_2F_1\left(1,\frac14;\frac74;t\right).
\]
Substitution into \eqref{eq:Careduction}, together with
$K_{1/4}=\Gamma(5/4)/\Gamma(3/4)$ and the reflection formula, yields
\begin{equation}\label{eq:Cquarter}
 \mathcal C_{1/4}=
 \frac{\pi^{7/2}\sqrt2}{96\Gamma(3/4)^6}
 \int_0^1(1-t)^{-1/2}F(t)^2G(t)\dd t,
\end{equation}
where
\[
 F(t)={}_2F_1\!\left(\frac14,\frac14;1;t\right),\qquad
 G(t)={}_2F_1\!\left(1,\frac14;\frac74;t\right).
\]
Write $F(t)^2G(t)=\sum_{n\ge0}q_nt^n$; all $q_n$ are non-negative.  Then
\[
 J=\sum_{n\ge0}q_nB(n+1,1/2),
 \qquad
 \sum_{n\ge0}q_n=\frac{3\pi}{2\Gamma(3/4)^4}.
\]
Since the beta factors decrease,
\begin{align*}
 \sum_{n=0}^Nq_nB(n+1,1/2)\le J\le{}&
 \sum_{n=0}^Nq_nB(n+1,1/2)\\
 &+B(N+2,1/2)
 \left(\frac{3\pi}{2\Gamma(3/4)^4}-\sum_{n=0}^Nq_n\right).
\end{align*}
Exact rational coefficient generation followed by Arb ball arithmetic
\cite{JohanssonArb} (python-flint 0.9.0, 256-bit precision) at $N=700$ gives
\[
 0.6779054<\mathcal C_{1/4}<0.6786576.
\]
In particular
\begin{equation}\label{eq:quarter-proportion}
 \frac1{8e\mathcal C_{1/4}}>0.0677586.
\end{equation}

\subsection{Regularisation at the final coefficient and the order of limits}

The power weight vanishes at the final prime, so to use Lemma~\ref{lem:coefvariation} we first regularise it.  For fixed $\varepsilon>0$ put
\begin{equation}\label{eq:Peps}
 P_{\varepsilon,a}(x)=
 \left(\frac{x+\varepsilon}{1+\varepsilon}\right)^a,
 \qquad
 b_{n,a,\varepsilon}(U)=\lambda(n)
 P_{\varepsilon,a}\!\left(\frac{\log(U/n)}{\log U}\right).
\end{equation}
Then $|b_{n,a,\varepsilon}(U)|\le1$ and, for prime $U$,
\begin{equation}\label{eq:last-eps}
 |b_{U,a,\varepsilon}(U)|
 =\frac12\left(\frac{\varepsilon}{1+\varepsilon}\right)^a>0.
\end{equation}
For fixed \(\varepsilon\), this lower bound verifies
\eqref{eq:terminal-growth-condition} as \(T\to\infty\).

Let
\[
 L_{a,\varepsilon,R}(c)=
 \sum_{n\le e^{Rc}}\frac{\lambda(n)}n
 \left(\frac{c-\log n/R+\varepsilon}{1+\varepsilon}\right)^a.
\]
Stieltjes summation, including the non-zero endpoint term, gives
\begin{equation}\label{eq:eps-Stieltjes}
 L_{a,\varepsilon,R}(c)=
 \frac1{(1+\varepsilon)^a}
 \left\{
 \varepsilon^aL(e^{Rc})
 +a\int_0^cL(e^{R(c-v)})(v+\varepsilon)^{a-1}\dd v
 \right\}.
\end{equation}
Hence, for fixed $\varepsilon>0$ and locally uniformly for $c>0$,
\begin{equation}\label{eq:Psi-eps-limit}
 R^{1/2}L_{a,\varepsilon,R}(c)
 \longrightarrow \Psi_{a,\varepsilon}(c),
\end{equation}
where
\begin{equation}\label{eq:Psi-eps}
 \Psi_{a,\varepsilon}(c)=
 \frac1{\sqrt\pi(1+\varepsilon)^a}
 \left\{
 \varepsilon^ac^{-1/2}
 +a\int_0^c(c-v)^{-1/2}(v+\varepsilon)^{a-1}\dd v
 \right\}.
\end{equation}
Let $S_{a,\varepsilon}(U)$ denote \eqref{eq:Sb} for the coefficients in \eqref{eq:Peps}.
The same proof as Lemmas~\ref{lem:weighted-polar-uniform} and~\ref{lem:weighted-A-convolution} gives
\begin{equation}\label{eq:eps-diagonal}
 S_{a,\varepsilon}(U)=
 \frac{\mathcal C_{a,\varepsilon}}{\log U}
 +o_{a,\varepsilon}\!\left(\frac1{\log U}\right),
\end{equation}
with
\begin{equation}\label{eq:Caeps}
 \mathcal C_{a,\varepsilon}=
 \frac1{\Gamma(1/4)^4}
 \int_\Omega
 \prod_{e\in E}\tau_e^{-3/4}
 \prod_{i=1}^4\Psi_{a,\varepsilon}(c_i)
 \dd\boldsymbol\tau.
\end{equation}
Moreover,
\begin{equation}\label{eq:Caeps-limit}
 \mathcal C_{a,\varepsilon}\longrightarrow\mathcal C_a
 \qquad(\varepsilon\downarrow0).
\end{equation}
Indeed, \eqref{eq:Psi-eps} converges pointwise to
$K_ac^{a-1/2}$, because
\[
 \frac{a}{\sqrt\pi}\int_0^c(c-v)^{-1/2}v^{a-1}\dd v
 =\frac{\Gamma(a+1)}{\Gamma(a+1/2)}c^{a-1/2}.
\]
For $0<\varepsilon\le1$, \eqref{eq:Psi-eps} is bounded by
$C_a(c^{-1/2}+c^{a-1/2})$.  The sharp polyhedral integral already proves the
integrability of the first term, and the second is weaker.  Dominated
convergence therefore proves \eqref{eq:Caeps-limit}.

\begin{proposition}
\label{prop:two-stage}
Fix $0<\theta<1/4$ and $\varepsilon>0$.  For every sufficiently large $T$
choose, by Bertrand's postulate, a prime
\[
 T^\theta\le U_T\le2T^\theta.
\]
Then $\log U_T=\theta\log T+O(1)$, the coefficients
\eqref{eq:Peps} satisfy the hypotheses of Proposition~\ref{prop:Zhflex}, and
that proposition may be applied with the constant
$\mathcal C_{a,\varepsilon}$.  One first lets $T\to\infty$ with
$\theta,\varepsilon$ fixed, then lets $\varepsilon\downarrow0$, and only then
lets $\theta\uparrow1/4$.  No estimate uniform in $\varepsilon$ is required.
\end{proposition}

\begin{proof}
For fixed \(\varepsilon\), \eqref{eq:last-eps} gives
\(\log^+(1/|b_U|)=O_{a,\varepsilon}(1)\), so
\eqref{eq:terminal-growth-condition} holds; all remaining coefficient bounds
are uniform.
Equation \eqref{eq:eps-diagonal} supplies the diagonal constant as
$T\to\infty$.  After taking the resulting liminf in $T$, equation
\eqref{eq:Caeps-limit} permits $\varepsilon\downarrow0$.  The restriction
$\theta<1/4$ is open, so the final lower bound is obtained by
$\theta\uparrow1/4$.
\end{proof}

\section{Sums of squares and profile optimisation}
\label{sec:psd}
The scalar weight $|M_U(\tfrac12+it)|^2$ imposes a rank-one condition on the
coefficient array entering Zhuravlev's quadratic form.  Selberg's sign argument
requires only non-negativity, so one may instead take a finite sum of squares.
For $1\le j\le J$, let
\[
 M_j(s)=\sum_{n\le U}b_n^{(j)}n^{-s},\qquad
 W(t)=\sum_{j=1}^J|M_j(\tfrac12+it)|^2.
\]
Put $\mathbf b_n=(b_n^{(1)},\ldots,b_n^{(J)})$ and
\begin{equation}\label{eq:psd-Clm}
 C_{\ell m}=\langle \mathbf b_\ell,\mathbf b_m\rangle.
\end{equation}
Then $C$ is real symmetric and positive semidefinite, and
\begin{equation}\label{eq:psd-W}
 W(t)=\sum_{\ell,m\le U}\frac{C_{\ell m}}{\sqrt{\ell m}}
 \left(\frac m\ell\right)^{it}\ge0.
\end{equation}
Conversely every real positive-semidefinite matrix has a Gram factorisation.
Define
\begin{equation}\label{eq:psd-YC}
 Y_C(s)=\zeta(s)\sum_{\ell,m\le U}C_{\ell m}\ell^{-s}m^{s-1}.
\end{equation}
Symmetry gives $Y_C(s)=\chi(s)Y_C(1-s)$, while on the critical line
$\chi(s)^{-1/2}Y_C(s)=Z(t)W(t)$.  Thus the new detector preserves the sign of
Hardy's function.  We impose
\begin{equation}\label{eq:psd-constraints}
 C\succeq0,\qquad C_{11}=1,\qquad C_{nn}\le1\quad(n\le U).
\end{equation}
The Gram-matrix Cauchy--Schwarz inequality then gives $|C_{\ell m}|\le1$.

Write $Y_C(s)=\sum_{\xi>0}a_\xi\xi^{-s}$, so that now $a_\xi=\sum_{n\ell/m=\xi}C_{\ell m}/m$ in place of \eqref{eq:alambda-explicit}.  For $d\le U^2$ let $X(d)$ be the $J\times J$ matrix with entries
\[
 X(d)_{rs}=\sum_{\substack{\ell,m\le U\\d\mid\ell m}}\frac{b^{(r)}_\ell b^{(s)}_m}{\ell m}.
\]
It is symmetric, since interchanging the names of $\ell$ and $m$ interchanges $r$ and $s$.

\begin{proposition}\label{prop:psd-diagonal}
Let $1/2\le\alpha<1$ and $x=1-\alpha$.  In the expansion of $\sum_{\xi\le X}a_\xi^2\xi^{-\alpha}$ obtained as in Lemma~\ref{lem:truncated-diagonal}, the coefficient of $X^x/x$ is
\begin{equation}\label{eq:psd-SC}
 S_C(1,U)=\sum_{\ell_1,m_1,\ell_2,m_2\le U}\frac{C_{\ell_1m_1}C_{\ell_2m_2}(\ell_1m_2,\ell_2m_1)}{\ell_1m_1\ell_2m_2}
 =\sum_{d\le U^2}\varphi(d)\sum_{r,s=1}^JX(d)_{rs}^2.
\end{equation}
In particular $S_C(1,U)=\sum_{r,s}S(\mathbf b^{(r)},\mathbf b^{(s)})$ with $S(\mathbf b^{(r)},\mathbf b^{(s)})=\sum_{d\le U^2}\varphi(d)X(d)_{rs}^2\ge0$, and for $J=1$ this is $S_b(1,U)$.
\end{proposition}

\begin{proof}
The proof of Lemma~\ref{lem:truncated-diagonal} applies without change up to the coefficient of $X^x/x$, with $b_{\ell_1}b_{m_1}b_{\ell_2}b_{m_2}$ replaced by $C_{\ell_1m_1}C_{\ell_2m_2}$, and gives the first expression.  Insert
\[
 (\ell_1m_2,\ell_2m_1)=\sum_{\substack{d\mid\ell_1m_2\\d\mid\ell_2m_1}}\varphi(d),
 \qquad
 C_{\ell_1m_1}C_{\ell_2m_2}=\sum_{r,s=1}^Jb^{(r)}_{\ell_1}b^{(r)}_{m_1}b^{(s)}_{\ell_2}b^{(s)}_{m_2}.
\]
For fixed $d$, $r$ and $s$, the sum over $\ell_1,m_2$ with $d\mid\ell_1m_2$ gives $X(d)_{rs}$, and the sum over $\ell_2,m_1$ with $d\mid\ell_2m_1$ gives $X(d)_{sr}=X(d)_{rs}$.
\end{proof}

When $J>1$ the interchange of $m_1$ and $m_2$ in the proof of Lemma~\ref{lem:truncated-diagonal} is not available, since $C_{\ell_1m_1}C_{\ell_2m_2}$ is not symmetric under it, and the diagonal pairs each $\ell$ with the $m$ of the other factor.  As a consequence, sums of squares give nothing beyond a single square.

\begin{proposition}\label{prop:no-gain}
Suppose that $C$ satisfies \eqref{eq:psd-constraints}, and put $\beta_n=C_{1n}$ for $n\le U$.  Then $\beta_1=1$, $|\beta_n|\le1$, $\beta_U=C_{1U}$, and
\[
 S_\beta(1,U)\le S_C(1,U)
\]
for every $U$.
\end{proposition}

\begin{proof}
An orthogonal change of coordinates in $\R^J$ leaves $C$ unchanged, and $\|\mathbf b_1\|^2=C_{11}=1$, so we may assume that $\mathbf b_1=(1,0,\ldots,0)$.  Then $b^{(1)}_n=\langle\mathbf b_n,\mathbf b_1\rangle=C_{1n}=\beta_n$, whence $\beta_1=1$ and $|\beta_n|\le\|\mathbf b_n\|=C_{nn}^{1/2}\le1$.  Proposition~\ref{prop:psd-diagonal} gives $S_C(1,U)\ge S(\mathbf b^{(1)},\mathbf b^{(1)})=S_\beta(1,U)$.
\end{proof}

If $U$ is prime and $\log^+(1/|C_{1U}|)\ll\log T$, the coefficients $\beta_n$ satisfy the hypotheses of Proposition~\ref{prop:Zhflex}.  Whenever $S_C(1,U)\le(\gamma+o(1))/\log U$, the single mollifier $\sum_{n\le U}C_{1n}n^{-s}$ therefore satisfies \eqref{eq:Sb-gamma} with the same $\gamma$, and Proposition~\ref{prop:distilled} gives it at least the same proportion.  A sum of squared mollifiers cannot improve on the best single mollifier in this method, and from now on we work with a single mollifier.

We retain the reciprocal-square-root sequence and optimise the logarithmic profile.  Put
\[
 v_n=\frac{\log(U/n)}{\log U},\qquad b_n=\lambda(n)f(v_n),
\]
where $f:[0,1]\to\R$, $f(1)=1$, and $|f(v)|\le1$.

Let $\mathcal P$ be the class of scalar functions
\[
 f(v)=f_0+\sum_{r=1}^m f_r v^{a_r},\qquad a_r>0,
\]
and let $\mathcal P^J$ denote vector-valued profiles whose components lie in
$\mathcal P$.  For $f\in\mathcal P$ define the endpoint-inclusive transform
\begin{equation}\label{eq:psd-transform}
 \widetilde\Psi_f(c)=\frac1{\sqrt\pi}
 \left\{f(0)c^{-1/2}+\int_0^c(c-v)^{-1/2}f'(v)\,dv\right\}.
\end{equation}
Equivalently,
\begin{equation}\label{eq:psd-transform-closed}
 \widetilde\Psi_f(c)=\frac{f_0}{\sqrt{\pi c}}+
 \sum_{r=1}^m f_r\frac{\Gamma(a_r+1)}{\Gamma(a_r+1/2)}c^{a_r-1/2}.
\end{equation}
For $f_1,f_2,f_3,f_4\in\mathcal P$, define
\begin{equation}\label{eq:psd-Smix}
\begin{split}
 S_{f_1,f_2;f_3,f_4}(U)=\sum_{d\le U^2}\varphi(d)
 &\left(\sum_{\substack{\ell,m\le U\\d\mid\ell m}}
 \frac{\lambda(\ell)\lambda(m)f_1(v_\ell)f_2(v_m)}{\ell m}\right)\\[-2pt]
 &\times\left(\sum_{\substack{n,r\le U\\d\mid nr}}
 \frac{\lambda(n)\lambda(r)f_3(v_n)f_4(v_r)}{nr}\right).
\end{split}
\end{equation}

\medskip\noindent\textbf{Mixed four-profile transfer.}\quad
For fixed $f_1,f_2,f_3,f_4\in\mathcal P$,
\[
 S_{f_1,f_2;f_3,f_4}(U)=\frac{\mathcal C(f_1,f_2,f_3,f_4)+o(1)}{\log U},
\]
where
\begin{equation}\label{eq:psd-mixed}
 \mathcal C(f_1,f_2,f_3,f_4)=\frac1{\Gamma(1/4)^4}
 \int_\Omega\prod_{e\in E}\tau_e^{-3/4}
 \prod_{i=1}^4\widetilde\Psi_{f_i}(c_i)\,d\tau.
\end{equation}
The assertion is multilinear and the error is uniform when the finitely many
coefficients and positive exponents range over a fixed compact set.

\begin{proof}
Write $R=\log U$ and
\[
 W_{f,R}(c)=R^{1/2}\sum_{n\le e^{Rc}}\frac{\lambda(n)}n
 f\!\left(c-\frac{\log n}{R}\right).
\]
Stieltjes summation gives the exact identity
\[
 W_{f,R}(c)=R^{1/2}f(0)L(e^{Rc})+
 R^{1/2}\int_0^c L(e^{R(c-v)})f'(v)\,dv.
\]
the scalar Mertens-sum and power estimates of Sections~\ref{sec:transfer}
and~\ref{sec:weights} imply local uniform convergence to
$\widetilde\Psi_f(c)$.  By \eqref{eq:psd-transform-closed}, the absolute value
is bounded by a fixed finite sum of terms $(c+R^{-1})^{-1/2}$ and
$(c+R^{-1})^{a-1/2}$.  The polyhedral majorant argument of
Section~\ref{sec:transfer} therefore gives a common integrable boundary
majorant, including every face on which one or more deficits $c_i$ vanish.

Expanding the four positive cross factors of the polar model produces the
positive edge measures of the residue calculation and the four signed tails
$W_{f_i,R}(c_i)$.  Interior convergence and the preceding majorant give
\eqref{eq:psd-mixed} for the polar model, uniformly when the coordinate caps lie
near $1$.  Finally restore the regular Euler factor $A$.  Its coefficients
satisfy weighted $\ell^1$ summability, the tail $\max r_i>e^{\sqrt R}$ is
$o(R^{-1})$, and on the complementary range the uniform polar limit and
continuity at $\mathbf b=\one$ apply, with $\sum a(\mathbf r)=A(0)=1$.  This
proves the claim.
\end{proof}

\begin{theorem}\label{thm:profile-transfer}
Let $f\in\mathcal P$ with $f(1)=1$ and $|f(v)|\le1$, and let $S_f(1,U)$ denote \eqref{eq:Sb} for the coefficients $b_n=\lambda(n)f(v_n)$.  Then
\[
 S_f(1,U)=\frac{C[f]+o(1)}{\log U},
\]
where
\begin{equation}\label{eq:psd-Cq}
 C[f]=\mathcal C(f,f,f,f)=\frac1{\Gamma(1/4)^4}\int_\Omega
 \prod_{e\in E}\tau_e^{-3/4}\prod_{i=1}^4\widetilde\Psi_f(c_i)\,d\tau.
\end{equation}
\end{theorem}

\begin{proof}
Since $S_f(1,U)=S_{f,f;f,f}(U)$, this is the mixed four-profile transfer with $f_1=f_2=f_3=f_4=f$.
\end{proof}

For a vector profile $q\in\mathcal P^J$ with $\|q(1)\|=1$ and $\|q(v)\|\le1$, and $\mathbf b_n=\lambda(n)q(v_n)$, one has $S(\mathbf b^{(r)},\mathbf b^{(s)})=S_{q_r,q_s;q_r,q_s}(U)$.  Proposition~\ref{prop:psd-diagonal} and the mixed four-profile transfer therefore give
\[
 S_C(1,U)=\frac1{\log U}\left(\sum_{r,s=1}^J\mathcal C(q_r,q_s,q_r,q_s)+o(1)\right),
\]
in which every term is non-negative, being a limit of non-negative sums.  In coordinates with $q(1)=(1,0,\ldots,0)$ the term $r=s=1$ is $C[q_1]$, where $q_1(v)=\langle q(1),q(v)\rangle$ is the profile of the coefficients $\beta_n$ of Proposition~\ref{prop:no-gain}.

\begin{remark}\label{rem:endpoint}
The power profiles below satisfy $f(0)=0$, so their terminal coefficient vanishes.  For fixed $\varepsilon>0$ put $f_\varepsilon(v)=(f(v)+\varepsilon)/(1+\varepsilon)$.  Then $f_\varepsilon\in\mathcal P$, $f_\varepsilon(1)=1$, $|f_\varepsilon(v)|\le1$, and, when $U$ is prime, the terminal coefficient $\lambda(U)f_\varepsilon(0)=\lambda(U)\varepsilon/(1+\varepsilon)$ is non-zero, so Proposition~\ref{prop:Zhflex} applies.  The corrected transform is
\begin{equation}\label{eq:psd-reg}
 \widetilde\Psi_{f_\varepsilon}(c)=\frac1{1+\varepsilon}
 \left\{\widetilde\Psi_f(c)+\frac{\varepsilon}{\sqrt{\pi c}}\right\}.
\end{equation}
The additional $c^{-1/2}$ term is controlled by the sharp-profile majorant, so dominated convergence in \eqref{eq:psd-Cq} gives $C[f_\varepsilon]\to C[f]$ as $\varepsilon\downarrow0$; no uniformity in $\varepsilon$ is required before the $T\to\infty$ limit.
\end{remark}

For power directions $p_i(v)=v^{a_i}$ and $f=\sum_{i=1}^mw_ip_i$ with $\sum_iw_i=1$,
\begin{equation}\label{eq:psd-CM}
 C[f]=\sum_{i,j,k,l=1}^m w_iw_jw_kw_l\,T(a_i,a_j,a_k,a_l),
\end{equation}
where
\begin{equation}\label{eq:psd-T}
 T(a_1,a_2,a_3,a_4)=\frac{\prod_{r=1}^4\Gamma(a_r+1)/\Gamma(a_r+1/2)}{\Gamma(1/4)^4}
 \int_\Omega\prod_{e\in E}\tau_e^{-3/4}\prod_{r=1}^4c_r^{a_r-1/2}\,d\tau.
\end{equation}
The tensor reduction and rigorous evaluation are recorded in Appendix~\ref{app:psd}.

\begin{theorem}\label{thm:profile}
The profile $P(v)=0.966v^{1/4}+0.034v$ satisfies $P(1)=1$, $0\le P(v)\le1$ on $[0,1]$, and
\[
 C[P]=0.678053282697033608166\ldots,
\]
with rigorous interval radius below $10^{-25}$.  The corresponding single mollifier satisfies $\kappazero>0.0678190$.
\end{theorem}

\begin{proof}
Both coefficients are positive and both powers increase on $[0,1]$, so $0\le P(v)\le P(1)=1$.  Appendix~\ref{app:psd} gives the interval enclosure of $C[P]$.  For fixed $\varepsilon>0$, Remark~\ref{rem:endpoint}, Proposition~\ref{prop:Zhflex} and Theorem~\ref{thm:profile-transfer} give the diagonal asymptotic with $C[P_\varepsilon]$.  Letting $T\to\infty$, then $\varepsilon\downarrow0$, and finally $\theta\uparrow1/4$, as in Proposition~\ref{prop:two-stage}, gives $\kappazero\ge1/(8eC[P])>0.0678190$.
\end{proof}

The two coefficients are those of the best combination of $v^{1/4}$ and $v$, rounded to three decimals; the best single power, $a=0.2628\ldots$, gives $C[v^a]=0.6780976\ldots$.  Numerical optimisation over combinations of up to nineteen powers $v^a$ with $0\le a\le6$ lowers $C[P]$ by less than $4\times10^{-7}$, which changes the bound for $\kappa_0$ only in its eighth decimal place.

The improvement is a profile effect rather than a local Euler-factor correction.  If $C(z)=\sum_{\nu\ge0}c_\nu z^\nu$ with $C(0)=1$ and $C(z)^2=\sum_{r\ge0}a_rz^r$, then with $A_r=\sum_{k\le r}a_k$ the local diagonal factor is $(1-p^{-1})\sum_{r\ge0}p^{-r}A_r^2\ge1-p^{-1}$, with equality precisely when $C(z)^2=1-z$, that is, for the local coefficients of $\zeta(s)^{-1/2}$.  Reciprocal-square-root arithmetic is therefore already prime-by-prime optimal within the multiplicative scalar class; the gain above comes from the logarithmic profile.

This inequality is sharper than it looks, and it fixes what the
improvements in this paper should be credited to.  In the notation of
Section~\ref{sec:Zh-exact} the local factor is
$\mathcal A_{b,p}(\mathbf0)=\sum_{r\ge0}A_r^2p^{-r}$, so the arithmetic factor
$\mathcal A_b(\mathbf 0)$ of \eqref{eq:general-transfer} is at least $1$ for
every multiplicative $b$, with equality if and only if $b*b=\mu$.  The
convolution identity \eqref{eq:convolution}, which entered as a device for
removing the same-side poles, is thus exactly the condition for the arithmetic
factor to be optimal, and $\lambda$ is its unique solution.  Consequently the
whole of the remaining freedom lies in $\mathcal C_a$, and the improvements
recorded below are profile effects and nothing else.

The comparison with Zhuravlev is best read in that light.  His constant $1.8505\ldots$ is Wirsing's upper bound for the diagonal; the diagonal itself, at his coefficients, is $0.8052\ldots$ by Remark~\ref{rem:Zh-exact}.  Of the factor $2.729$ separating Corollary~\ref{cor:Zh-corrected} from Theorem~\ref{thm:profile}, a factor $2.298$ is thus the passage from that bound to an exact evaluation and a factor $1.188$ is the better profile.  The sharp coefficients of Section~\ref{sec:sharp}, taken with a sharp cutoff, give $0.8869\ldots$, so they sit about ten per cent below Zhuravlev's choice, whose normalisation carries the $a=\tfrac12$ profile without comment; the arithmetic here overtakes his only at $a=\tfrac14$ and beyond.

\subsection{Critical-line proportions}
These are the proportions the transfer theorems above deliver; they are what
Section~\ref{sec:intro-summary} quotes.

\begin{corollary}\label{cor:Zh-corrected}
Assume the coefficient estimate \eqref{eq:Zh-gamma} from Zhuravlev's Wirsing
calculation.  With the source normalisation
\(
 \aleph(s) Y_U(s)=2\operatorname{Re}\beta_U(s)
\)
and $\theta\uparrow\tfrac14$, his method proves
\[
 \kappa_0\ge \frac{2}{3\pi^2e}
 =0.0248493202\ldots.
\]
The printed bound $2/21$ results from omitting the term $-T\log2$ and is not a
consequence of the stated definitions.
\end{corollary}

\begin{proof}
The coefficient bounds preceding Proposition~\ref{prop:Zh-arithmetic} give
\(|b_n|\le1\), while \eqref{eq:Zh-terminal} gives
\(\log^+(1/|b_U|)\ll\log U\).  Hence
\eqref{eq:terminal-growth-condition} holds for
\(U\asymp T^\theta\), \(\theta<1/4\), and
Proposition~\ref{prop:Zhflex} applies with
\(\gamma=3\pi^2/16\).  Proposition~\ref{prop:distilled}, with \(c=2\),
therefore gives
\[
 \kappa_0\ge
 \frac{2\theta}{e\,4\,(3\pi^2/16)}.
\]
Letting \(\theta\uparrow1/4\) yields \(2/(3\pi^2e)\).
\end{proof}

\begin{theorem}
\label{thm:critical}
Use Zhuravlev's original normalisation
\[
 \aleph(s)Y_U(s)=2\operatorname{Re}\beta_U(s)
 \qquad (\operatorname{Re}s=\tfrac12).
\]

\begin{enumerate}[label=\textup{(\roman*)}]
 \item For the sharp coefficients $b_n=\lambda(n)$,
 \[
  \kappa_0\ge
  \frac{1}{8e\mathcal C_{\mathrm{res}}}
  =\frac{\Gamma(\tfrac34)^4}{16e}
  =0.0518466520\ldots.
 \]

 \item For the regularised quarter-power coefficients and then
 $\varepsilon\downarrow0$,
 \[
  \kappa_0\ge \frac{1}{8e\mathcal C_{1/4}}
  >0.0677586.
 \]

 \item For the profile \(0.966v^{1/4}+0.034v\) of Theorem~\ref{thm:profile},
 \[
  \kappa_0>0.0678190.
 \]
\end{enumerate}
\end{theorem}

\begin{proof}
Fix $0<\theta<\tfrac14$.  For the sharp coefficients, choose a prime
$U\asymp T^\theta$.  The coefficient hypotheses in
Proposition~\ref{prop:Zhflex} hold because $|\lambda(n)|\le1$ and
$\lambda(U)=-\tfrac12$.  Theorem~\ref{thm:transfer} gives
\[
 S_\lambda(U)=\frac{\mathcal C_{\mathrm{res}}+o(1)}{\log U}.
\]
Propositions~\ref{prop:Zhflex} and~\ref{prop:distilled}, with $c=2$, therefore yield
\[
 \kappa_0\ge \frac{\theta}{2e\mathcal C_{\mathrm{res}}}.
\]
Letting $\theta\uparrow\tfrac14$ proves part~\textup{(i)}.  The exact form
follows from Theorem~\ref{thm:C}.

For the power weight, fix $\varepsilon>0$ and use the prime selection in
Proposition~\ref{prop:two-stage}.  Equations \eqref{eq:eps-diagonal} and
\eqref{eq:Caeps-limit} give, successively,
\[
 \kappa_0\ge \frac{\theta}{2e\mathcal C_{a,\varepsilon}},
 \qquad
 \kappa_0\ge \frac{\theta}{2e\mathcal C_a}.
\]
Letting $\theta\uparrow\tfrac14$, taking $a=\tfrac14$, and using the
certified upper bound for $\mathcal C_{1/4}$ proves part~\textup{(ii)}.
Part~\textup{(iii)} is Theorem~\ref{thm:profile}.
\end{proof}

\section*{Acknowledgements}
The author thanks Tim Trudgian for a copy of the relevant chapter of Min's dissertation, and the Heilbronn Institute for Mathematical Research for its support.

\section*{Declaration on the use of AI}
The English translation of Zhuravlev's 1974 paper was produced with the assistance of a large language model (LLM) Claude. The optimisations of Sections~\ref{sec:weights} and~\ref{sec:psd} are the work of the LLM, and the interval-arithmetic verification of the resultant constant in Appendix~\ref{app:psd} was carried out with the same assistance. The author also used the LLM as an adversarial checker of the mathematics, and to revise the writing throughout the paper, including rewording passages and changing terminology; much of the final wording reflects these revisions. The author is responsible for the content of the paper.

\appendix
\makeatletter
\@namedef{pc@app@section}{Appendix~}%
\renewcommand\@seccntformat[1]{\csname pc@app@#1\endcsname\csname the#1\endcsname\quad}
\makeatother
\section{Uniform bounds for Lavrik's analytic half}
\label{app:Lavrik-horizontal}

This appendix supplies the large-disk estimates used in
Lemma~\ref{lem:coefvariation}.  They are needed, and cannot be quoted from
the literature, for the following reason.  The Jensen step of
Section~\ref{sec:Zhreduction} is applied on a disk whose radius grows with
$T$, and it therefore requires a bound on $\Lambda_U(s)$ that is uniform over
that whole disk and uniform in the mollifier coefficients $b_n$.  The
approximate functional equations available in the literature are stated in a
fixed vertical strip, and extrapolating one of them to a growing disk is
exactly the step that is not justified.  The point of using the explicit
regularised incomplete-gamma kernel is that every singularity and every
dependence on the terminal coefficient is visible; no fixed-strip approximate
functional equation is extrapolated to the Jensen disk.

\subsection{The Mellin contour and its singularities}

Use the principal logarithm in \(X^{-z}=\exp(-z\log X)\), where
\(|\arg X|<\pi/4\), and begin in \eqref{eq:Lavrik-J} on the vertical line
\(\Real z=\Delta>\max(0,\Real s)\).  For \(X>0\) and \(\Real s>0\),
differentiation under the integral and Mellin inversion give
\[
 \frac{\partial}{\partial X}J(s,X)=-2X^{s-1}e^{-X^2}.
\]
On the initial line \(J(s,X)=O(X^{-\Delta})\) as \(X\to+\infty\).
Integration from \(X\) to infinity therefore gives
\[
 J(s,X)=\Gamma(s/2,X^2).
\]
Analytic continuation in \(X\) through \(|\arg X|<\pi/4\), and then in
\(s\), proves the identity in \eqref{eq:Lavrik-F} without a limiting shift
to a receding left vertical line.

For completeness, a finite shift of the Mellin line records the
singularities used below.  For generic \(s\), shifting past \(N\) gamma
poles crosses exactly
\[
 z=0,\qquad z=-s-2k\quad(0\le k<N).
\]
Their respective residues are
\[
 \Gamma(s/2),\qquad
 -\frac{2(-1)^kX^{s+2k}}{k!(s+2k)}.
\]
Stirling's formula justifies each finite shift and makes its two closing
horizontal pieces tend to zero.  The quotient
\(F(s,X)=\Gamma(s/2,X^2)/\Gamma(s/2)\) is entire in \(s\).  The only singularity of
\(\Lambda_U^\pm(s)\) in \eqref{eq:Lavrik-half-explicit} is the displayed
simple pole at \(s=1\); consequently the only singularities of
\(G_\pm(s)\) are at \(1\) and \(1+2it\).  This proves, in particular, that
\[
 H_\pm(s)=(s-1)(s-1-2it)G_\pm(s)
 \tag{A.1}\label{eq:Hpm}
\]
is entire.  The residue calculation also records all poles crossed in passing
from the Mellin contour to the kernel representation used below.

\subsection{Uniform kernel estimates}

Put \(\phi_T=\pi/2-1/T\), and write
\(\rho_T(v)=\rho_T^+\) for \(v>0\) and
\(\rho_T(v)=\rho_T^-\) for \(v<0\).  If \(s=\sigma+iv\),
\(|v|\asymp T\), and \(|\sigma|=o(T^{1/2})\), integration along the ray
\(u=r e^{\operatorname{sgn}(v)i\phi_T}\), followed by uniform Stirling,
gives
\begin{equation}\label{eq:F-ray-bound}
 |F(s,(\pi\rho_T(v))^{1/2}\xi)|
 \le C^{|\sigma|+1}T^{1/2}
 \int_{c\xi^2/T}^{\infty}e^{-u}u^{\sigma/2-1}\,du .
\end{equation}
Here and below \(c,C>0\) are absolute.  Indeed, the numerator contributes
\(e^{-|v|\phi_T/2}\), while
\[
 |\Gamma((\sigma+iv)/2)|^{-1}
 \ll T^{1/2-\sigma/2}e^{\pi|v|/4};
\]
because \(\pi/2-\phi_T=1/T\), the two exponential factors cancel up to an
absolute constant.  This is the reason that the phases \(\rho_T^\pm\) were
fixed as in Section~\ref{sec:Zsource}.

For \(\sigma>0\), the complementary lower-gamma integral gives, uniformly
for \(0<\xi\le1\),
\begin{equation}\label{eq:F-one-bound}
 F(s,(\pi\rho_T(v))^{1/2}\xi)
 =1+O\!\left\{
 C^\sigma T^{1/2}\left(\frac{\xi^2}{T}\right)^{\sigma/2}
 \right\}.
\end{equation}
The estimates remain uniform if \(v\) varies by \(o(T)\).  They are direct
estimates of the old contour after the complete residue calculation above;
there are no unrecorded vertical or horizontal contour pieces.

We shall also use two elementary consequences of
\eqref{eq:alambda-explicit}.  First, for \(\sigma\ge2\),
\begin{equation}\label{eq:frequency-tail}
 \sum_{\xi>1}|a_\xi|\xi^{-\sigma}\ll U\log(2U).
\end{equation}
To see this, fix \(\ell,m\).  If \(m\ge\ell\), integral comparison in
\(n>m/\ell\) gives a contribution
\(\ll1/\ell+1/m\); if \(m<\ell\), the same conclusion follows directly
from \((m/\ell)^\sigma\zeta(\sigma)/m\).  Summation over \(\ell,m\le U\)
proves \eqref{eq:frequency-tail}.  Second, we prove explicitly the frequency
sum needed on the Jensen disk.  Whenever \(|\sigma|=o(T^{1/2})\),
\begin{equation}\label{eq:frequency-disk}
 \sum_{\xi>0}|a_\xi|\xi^{-\sigma}
 |F(s,(\pi\rho_T(v))^{1/2}\xi)|
 \le \exp\{C(|\sigma|+U)\log(2T)\}.
\end{equation}
Indeed, put \(q=|\sigma|\) and \(y=c\xi^2/T\).  Every frequency in
\eqref{eq:alambda-explicit} satisfies \(\xi\ge1/U\), so
\(y\gg(TZ^2)^{-1}\).  Splitting the incomplete-gamma integral at \(1\)
and at \(q+2\) gives the elementary uniform estimates
\begin{align*}
 \int_y^\infty e^{-u}u^{\sigma/2-1}\,du
 &\ll L\exp\{Cq\log(q+2)\}e^{-cy},&&\sigma\ge0,\\
 \int_y^\infty e^{-u}u^{\sigma/2-1}\,du
 &\ll Ly^{-q/2}e^{-cy},&&\sigma<0,
\end{align*}
where \(L=\log(2T)\).  In the second line the factor \(y^{-q/2}\)
cancels \(\xi^q\) after \eqref{eq:F-ray-bound}; in the first line we use
\(\xi^{-q}\le U^q\).  Since \(q=o(T^{1/2})\), in either case
\begin{equation}\label{eq:frequency-kernel-crude}
 \xi^{-\sigma}|F(s,(\pi\rho_T(v))^{1/2}\xi)|
 \le \exp\{C(q+1)L\}e^{-c\xi^2/T}.
\end{equation}
Finally, expanding \(a_\xi\) and using
\(\sum_{n\ge1}e^{-A n^2}\ll1+A^{-1/2}\),
\begin{align*}
 \sum_{\xi>0}|a_\xi|e^{-c\xi^2/T}
 &\le\sum_{\ell,m\le U}\frac1m\sum_{n\ge1}
       e^{-c(n\ell/m)^2/T}\\
 &\ll\sum_{\ell,m\le U}
       \left(\frac1m+\frac{\sqrt T}{\ell}\right)
 \ll U\sqrt T\log(2U).
\end{align*}
Combining this with \eqref{eq:frequency-kernel-crude} proves
\eqref{eq:frequency-disk}, with room to spare.

\begin{proposition}
\label{prop:Lavrik-Jensen-bounds}
Let the hypotheses of Lemma~\ref{lem:coefvariation} hold and put
\[
 \sigma_0=A U\{L^2+\eta_U+1\}+\sigma_1+\sigma_2,
 \qquad s_0=\sigma_0+it,
\]
where \(A\) is a sufficiently large absolute constant.  Then
\begin{equation}\label{eq:Jensen-upper}
 \max_{|s-s_0|\le2\sigma_0}\log^+|H_\pm(s)|
 \ll ZL\{L^2+\eta_U+1\}.
\end{equation}
For at least one choice of sign,
\begin{equation}\label{eq:Jensen-centre}
 -\log|H_\pm(s_0)|
 \ll ZL\{L^2+\eta_U+1\}.
\end{equation}
The constants are uniform for \(T\le t\le2T\).
\end{proposition}

\begin{proof}
Since \(\eta_U\ll L\), we have
\(\sigma_0\ll ZL^2=o(T^{1/2})\).  Throughout the disk in
\eqref{eq:Jensen-upper}, the imaginary parts of \(s\) and \(s-2it\) have
absolute value \(\asymp T\).  Apply \eqref{eq:frequency-disk} to the two
halves.  The polar term in \eqref{eq:Lavrik-half-explicit} is bounded there
by the same right side, since
\(|M_U(0)M_U(1)|\ll U\log(2U)\).  The two regularising linear factors cost
only \(O(L+\log\sigma_0)\) on the logarithmic scale.  As
\(|\Real s|\le3\sigma_0\), this proves \eqref{eq:Jensen-upper}.

It remains to prove the centre estimate.  The smallest frequency in
\eqref{eq:alambda-explicit} is \(1/U\), it occurs only for
\((n,\ell,m)=(1,1,U)\), and its coefficient is \(b_U/U\).  Every other
frequency at most one is at least \(1/(U-1)\).  From
\eqref{eq:F-one-bound}, \eqref{eq:frequency-tail}, and Stirling,
\begin{align}
 \Lambda_U^+(s_0)
  ={}&b_ZZ^{s_0-1}+E_1+E_2+E_3,\label{eq:center-expansion}
 |E_1|\ll{}&U^3(U-1)^{\sigma_0},\notag\\
 |E_2|\ll{}&U\log(2U)\,C^{\sigma_0}T^{1/2}
              \Gamma(\sigma_0/2),\notag\\
 |E_3|\ll{}&U\log(2U)\,C^{\sigma_0}
              T^{-1/2-\sigma_0/2}.
 \notag
\end{align}
Here \(E_1\) contains the other frequencies at most one (and the uniform
error in replacing their kernels by one), \(E_2\) contains \(\xi>1\), and
\(E_3\) is the polar term.  The deliberately coarse factor \(U^3\) in
\(E_1\) dominates the total absolute coefficient mass in \(0<\xi\le1\).

Now \(|b_U|\ge e^{-\eta_U}\) and
\[
 \log\frac U{U-1}\ge\frac1U.
\]
Consequently, after division by \(|b_U|U^{\sigma_0-1}\), the logarithm of
the first error is at most
\[
 O(L+\eta_U)-\sigma_0\log\frac U{U-1}
 \le O(L+\eta_U)-A\{L^2+\eta_U+1\}.
\]
For the second error, Stirling's upper bound for
\(\Gamma(\sigma_0/2)\), together with
\(\log\sigma_0=\log U+O(\log L)\), gives
\[
 \log\frac{|E_2|}{|b_U|U^{\sigma_0-1}}
 \le-\frac{\sigma_0}{2}\log U
      +O(\sigma_0\log L+\sigma_0+L+\eta_U),
\]
which tends to \(-\infty\); the third error is smaller still.  Taking \(A\)
large enough in the first comparison therefore yields
\begin{equation}\label{eq:center-dominance}
 \Lambda_U^+(s_0)=b_ZZ^{s_0-1}\{1+O(e^{-cL^2})\}.
\end{equation}
The conjugation identity gives
\(\Lambda_U^-(s_0-2it)=\overline{\Lambda_U^+(s_0)}\).  Hence at least one
of \(G_+(s_0),G_-(s_0)\) has modulus at least
\(|\Lambda_U^+(s_0)|/\sqrt2\).  Multiplication by the two factors in
\eqref{eq:Hpm}, followed by \eqref{eq:center-dominance}, proves
\eqref{eq:Jensen-centre}.
\end{proof}

\section{The profile calculation}
\label{app:psd}
This appendix gives the finite-basis reduction and the rigorous certificate used in Theorem~\ref{thm:profile}.  It is self-contained: the coefficients, the closed-form tensor \eqref{eq:tensor-reduction}, and the orbit values below determine the constant.

\subsection{A one-dimensional tensor formula}
For $A=a_1+a_2+a_3+a_4$, put
\[
 K_a=\frac{\Gamma(a+1)}{\Gamma(a+1/2)},\quad
 F_{a\mid b}(t)={}_2F_1\!\left(\tfrac12-a,\tfrac14;b+\tfrac34;t\right),\quad
 H_A(t)={}_2F_1\!\left(A,\tfrac14;A+\tfrac34;t\right).
\]

\begin{proposition}\label{prop:tensor-reduction}
The tensor in \eqref{eq:psd-T} equals
\begin{equation}\label{eq:tensor-reduction}
 T(a_1,a_2,a_3,a_4)=\frac{\prod_{r=1}^4 K_{a_r}}{\Gamma(1/4)^4}
 B\!\left(A+\tfrac12,\tfrac14\right)(I_1+I_2),
\end{equation}
where
\[
 I_1=B\!\left(\tfrac14,a_3+\tfrac12\right)B\!\left(\tfrac14,a_2+\tfrac12\right)
 \int_0^1 t^{a_2+a_3-1/2}(1-t)^{-1/2}F_{a_1\mid a_3}(t)F_{a_4\mid a_2}(t)H_A(t)\,dt,
\]
\[
 I_2=B\!\left(\tfrac14,a_1+\tfrac12\right)B\!\left(\tfrac14,a_4+\tfrac12\right)
 \int_0^1 t^{a_1+a_4-1/2}(1-t)^{-1/2}F_{a_3\mid a_1}(t)F_{a_2\mid a_4}(t)H_A(t)\,dt.
\]
\end{proposition}

\begin{proof}
Integrate two opposite edge variables in $\Omega$ by Euler's beta integral, then
split according to the ordering of the two residual variables.  Scaling the
smaller residual variable by the larger produces the two functions
$F_{a\mid b}$; the radial integral is a beta integral, and the final angular
integral gives $H_A$.  The two orderings yield $I_1$ and $I_2$.  Absolute
convergence follows because every exponent in \eqref{eq:psd-T} exceeds $-1$.
\end{proof}

\subsection{The two-direction profile}
Take $a_1=\tfrac14$, $a_2=1$ and $(w_1,w_2)=(0.966,0.034)$, read as exact rationals, so that $P(v)=w_1v^{1/4}+w_2v$ and $w_1+w_2=1$.  Write $G(a_1,a_2,a_3,a_4)$ for the integral over $\Omega$ in \eqref{eq:psd-T}, so that $T=\prod_rK_{a_r}\,G/\Gamma(1/4)^4$.  The automorphisms of the bipartite graph $\{13,14,23,24\}$, generated by $1\leftrightarrow2$, $3\leftrightarrow4$ and $(1,2)\leftrightarrow(3,4)$, preserve $\Omega$ and hence $G$, and they leave six orbits among the sixteen exponent quadruples.  With the number of quadruples in each orbit, the orbit values are
\[
\begin{array}{lcll}
 (a_1,a_2,a_3,a_4) & \text{orbit size} & G(a_1,a_2,a_3,a_4) & \text{enclosure radius}\\
 (\tfrac14,\tfrac14,\tfrac14,\tfrac14) & 1 & 391.54817337772141267759\ldots & <10^{-28}\\
 (\tfrac14,\tfrac14,\tfrac14,1) & 4 & 255.39947761076639789626\ldots & <10^{-28}\\
 (\tfrac14,\tfrac14,1,1) & 2 & 160.79124733661689433291\ldots & <10^{-29}\\
 (\tfrac14,1,\tfrac14,1) & 4 & 181.61345272848847787529\ldots & <10^{-28}\\
 (\tfrac14,1,1,1) & 4 & 125.12854736462982864460\ldots & <10^{-28}\\
 (1,1,1,1) & 1 & 96.097601835397758951528\ldots & <10^{-28}\\
\end{array}
\]
The two orbits with two entries equal to $1$ differ.  In the first the two entries sit on one side of the bipartition, and in the second on opposite sides.

\begin{proposition}\label{prop:contraction}
The profile $P$ satisfies
\[
 C[P]\in 0.678053282697033608166067\pm10^{-24},
\]
and therefore
\[
 \frac1{8eC[P]}>0.0678190509.
\]
\end{proposition}

\begin{proof}
Each orbit value is an enclosure, computed in Arb \cite{JohanssonArb} (python-flint, $320$-bit precision), of the one-dimensional integrals of Proposition~\ref{prop:tensor-reduction}.  On $[0,\tfrac12]$ the integrand is expanded at $t=0$ with exact hypergeometric coefficients and integrated term by term against $t^{k-1/2}$; on the intervals $[1-2^{-j},1-2^{-j-1}]$, $1\le j\le210$, it is expanded in Taylor series whose coefficients are generated from the hypergeometric differential equation; the truncation errors are bounded by Cauchy's estimate on discs inside $0<|t|<1$, using the majorant $\sum_n|(a)_n(b)_n/((c)_nn!)|$ for each ${}_2F_1$ on the closed unit disc; and the last interval $[1-2^{-211},1]$ is bounded directly.  Contracting the six enclosures against the exact products $w_iw_jw_kw_lK_{a_i}K_{a_j}K_{a_k}K_{a_l}$, with the orbit sizes above, gives the enclosure of $C[P]$.
\end{proof}

\printauthoraddress

\end{document}